\documentclass[11pt,reqno]{amsart}
\usepackage{tikz,a4wide}
\usetikzlibrary{arrows.meta}
\usetikzlibrary{math,calc}
\usepackage{float}
\usepackage{amsmath}
\usepackage{xfrac}
\usepackage{stmaryrd}
\SetSymbolFont{stmry}{bold}{U}{stmry}{m}{n}
\usepackage{mathrsfs,bm,amsthm,mathtools,yfonts,amssymb,color}
\usepackage{esint}
\usepackage{xcolor}
\usepackage[pagebackref,citecolor=black,linkcolor=blue,colorlinks=true]{hyperref}
\usepackage{tikz-3dplot}
\usepackage{epstopdf}
\usepackage{cleveref}
\usepackage{csquotes} 
\MakeOuterQuote{"} 

\makeatletter 
\DeclarePairedDelimiter{\@tmpabs}{\lvert}{\rvert}
\newcommand{\@absstar}[1]{{\@tmpabs*{#1}}}
\newcommand{\@absnostar}[2][]{{\@tmpabs[#1]{#2}}}
\newcommand{\abs}{\@ifstar\@absstar\@absnostar}
\newcommand{\R}{\mathbb R}
\newcommand{\Capa}{\textrm{Cap}}
\renewcommand{\H}{\mathcal{H}}
\newcommand{\spt}{{\rm spt}\,}
\newcommand{\Z}{{\mathbb Z}}
\newcommand{\T}{{\mathbb T}}
\newcommand{\E}{{\mathbb E}}
\renewcommand{\exp}{{\rm exp}\,}
\makeatother

\makeatletter 
\DeclarePairedDelimiter{\@tmpnorm}{\lVert}{\rVert}
\newcommand{\@normstar}[1]{{\@tmpnorm*{#1}}}
\newcommand{\@normnostar}[2][]{{\@tmpnorm[#1]{#2}}}
\newcommand{\norm}{\@ifstar\@normstar\@normnostar}
\makeatother

\usepackage{cite}

\begingroup
\newtheorem{theorem}{Theorem}[section]
\newtheorem{lemma}[theorem]{Lemma}
\newtheorem{proposition}[theorem]{Proposition}
\newtheorem{corollary}[theorem]{Corollary}
\endgroup

\theoremstyle{definition}
\newtheorem{definition}[theorem]{Definition}
\newtheorem{remark}[theorem]{Remark}

\numberwithin{equation}{section}

\title{Level sets of fractional Sobolev functions}

\author[C. De Lellis]{Camillo De Lellis}
\address{School of Mathematics, Institute for Advanced Study, 1 Einstein Dr., Princeton NJ 08540, USA}
\address{Gran Sasso Science Institute, viale Francesco Crispi, 7, 67100 L’Aquila, Italy} 
\email{camillo.delellis@ias.edu}
\author[M. Chang]{Ming-Yuan Chang}
\address{Department of Mathematics, Princeton University, Washington Road, Princeton, NJ, 08540, USA}
\email{mc5366@princeton.edu}
\author[S. Mayboroda]{Svitlana Mayboroda}
\address{Department of Mathematics, ETH Z\"urich, R\"amistrasse 101, 8092 Z\"urich, Switzerland}
\address{School of Mathematics, Institute for Advanced Study, 1 Einstein Dr., Princeton NJ 08540, USA}
\email{svitlana.mayboroda@math.ethz.ch}

\begin{document}

\begin{abstract}
We prove a coarea-type result for scalar functions $f$ in fractional Sobolev spaces $W^{s, p} (\Omega)$ with $\Omega\subset \mathbb R^n$, $0<s<1$, and $1\leq p < \infty$. Our theorem shows that a.e. level set has zero Hausdorff $\mathcal{H}^{n-s}$ measure, where the level set $f^{-1} (y)$ is defined as the set all points at which $y$ is between the $\liminf$ and the $\limsup$ (as $r\downarrow 0$) of the averages of $f$ over the balls $B_r (y)$. A quite general construction of random series of wavelets shows also that with probability $1$ (many) level sets have indeed dimension $n-s$.  
\end{abstract}

\maketitle

\tableofcontents

\section{Introduction} 

In this note we are concerned with the dimension of the level sets of scalar functions $f$ in fractional Sobolev spaces $W^{s,p}$ with $0<s<1$ and $1\leq p \leq \infty$. Our motivation comes from turbulence theory. In particular there are experimental and numerical evidences that certain ``fronts'', which can be modeled as level sets of scalars advected by a turbulent flow, have certain specific fractal dimensions. Moreover there seems to be a relation between the dimension and the expected regularity of the advected scalar. In ``homogeneous'' turbulence it is expected that the advected scalar is H\"older continuous. However, the typical turbulence is ``intermittent'' and the expected regularity is rather in some (fractional) Sobolev class, thus motivating our interest. 

A folklore fact in the literature is that level sets of H\"older functions $C^s$ have codimension $s$, at least if both the function and level set are ``typical'' in some appropriate sense. 
A possible rigorous statement for a dimensional upper bound is a coarea-type inequality and in fact the following statement is a simple consequence of classical tools in geometric measure theory and follows from Federer \cite{Federer} (see Appendix \ref{a:coarea}). Here $[f]_s$ denotes the H\"older seminorm of $f$ and $\mathcal{H}^{n-s}$ the usual Hausdorff measure. 

\begin{theorem}\label{t:coarea-Hoelder}
For every $0\leq s \leq 1$ and every $1\leq n$ there is a constant $C(s,n)$ with the following property. If $Q_1 = (0,1)^n\subset \mathbb R^n$ and $f\in C^s (Q_1)$, then 
\begin{equation}\label{e:coarea-Hoelder}
\int_{\mathbb R} \mathcal{H}^{n-s} (\{f=y\})\, dy \leq C [f]_s\, .
\end{equation} 
\end{theorem}

Clearly this is just an extension of the classical coarea formula for Lipschitz functions. Our generalization of the above theorem to the fractional Sobolev setting yields an interesting ``twist''. Before stating it we need to define properly the level set of $f$, which does not necessarily have a continuous representative. 

\begin{definition}\label{d:level-set}
Given an $L^1$ function $f: \mathbb R^n \supset \Omega \to \mathbb R$, with $\Omega$ open, we define its level set $\{f=y\}$ as the set of points $x$ such that 
\[
\liminf_{r\downarrow 0} \frac{1}{|B_r (x)|} \int_{B_r (x)} f \leq y
\leq \limsup_{r\downarrow 0} \frac{1}{|B_r (x)|} \int_{B_r (x)} f \, .
\] 
\end{definition}

\begin{theorem}\label{t:coarea-fractional}
Assume $0<s<1$, $1\leq p < \infty$, and $1\leq n$. If $\Omega\subset \mathbb R^n$ is open and $f\in W^{s,p} (\Omega)$, then
\begin{equation}\label{e:coarea-fractional}
\mathcal{H}^{n-s} (\{f=y\}) = 0 \qquad \mbox{for a.e. $y$}\, .
\end{equation}
\end{theorem}

A lower bound on the dimension of the level sets requires some ``typicality'' assumption. In fact a smooth function $g\in C^\infty (B_1)$ is certainly also a H\"older function in $C^s (B_1)$ for all $s$, and for $g$ the classical coarea formula implies that the level set $\{g=y\}$ has dimension at most $n-1$ for a.e. $y$. On the other hand for appropriate random Fourier series on the torus $\mathbb T^n$, it is known since the seventies that the level sets have the expected dimension. More precisely, there is a variety of ``natural'' probability measures $\mathbb P$ over continuous functions on $\mathbb R^n$ for which
\begin{itemize}
\item[(a)] almost surely the Fourier series belongs to $C^\alpha$ for every $\alpha <s$,
\item[(b)] almost surely it does not belong to $C^\alpha$ for every $\alpha>s$.
\end{itemize}
Under mild assumptions on $\mathbb P$ the monograph \cite{Kahane} of Kahane shows then that the Hausdorff dimension of almost all level sets of random Fourier series, properly normalized, is, $\mathbb P$-almost surely, precisely $n-s$. In Section \ref{s:random} we will provide an analogous statement for suitable random series of {\em wavelets} in the fractional Sobolev case. Using wavelets seems natural if one is interested in constructing probability measures with properties (a) and (b) when we substitute $C^s$ with $W^{s,p}$. It also goes well with intermittency. For the precise statements we need to introduce a few concepts and some terminology and we refer therefore to Section \ref{s:random}. However, a direct corollary of these constructions is that the dimension $n-s$ in Theorem \ref{t:coarea-fractional} is optimal for the ``typical level set of a typical $W^{s,p}$ function''. 

\subsection{Acknowledgments} The three authors acknowledge the support of the Simons Foundation through the Simons Initiative for the Geometry of Flows s (Grant Award ID BD-
Targeted-00017375 CDeL, SM).

\section{Hausdorff measure of the graph and preliminaries}

Theorem \ref{t:coarea-fractional} will in fact be inferred from a stronger result. In order to state it, it is convenient to introduce the precise representative $\bar f$ of the function $f$ in Theorem \ref{t:coarea-fractional}.

\begin{definition}\label{d:precise-representative}
Let $f$ be as in Theorem \ref{t:coarea-fractional}.
As usual, a Lebesgue point for $f$ is a point $x$ for which there exists a real number $\bar{f} (x)$ such that
\begin{equation}\label{e:Lebesgue}
\lim_{r\downarrow 0} \frac{1}{r^n} \int_{B_r (x)} |f(y)-\bar{f} (x)|\, dy = 0\, . 
\end{equation}
The set of Lebesgue points will be called $Z(f)$. 
\end{definition}

\begin{theorem}\label{t:graph}
Let $f$ be as in Theorem \ref{t:coarea-fractional} and denote by $\Gamma (f)$ be the graph of $\bar{f}$, i.e.
\[
\Gamma (f) := \bigl\{(x, \bar{f} (x)) : x\in Z (f) \bigr\} \subset \mathbb R^{n+1}\, .
\]
Then $\mathcal{H}^{n+1-s} (\Gamma (f)) = 0$.
\end{theorem}

The latter result is stronger than Theorem \ref{t:coarea-fractional}.

\begin{proof}[Proof of Theorem \ref{t:coarea-fractional}] 
It is well known that $\mathcal{H}^{n-s} (\Omega \setminus Z(f)) =0$, cf. Theorem \ref{t:capacity} and Remark \ref{r:capacity-Hausdorff} below. On the cylinder $\Omega \times \mathbb R$ we introduce the orthogonal projection $\pi$ on the vertical axis, which is a Lipschitz map with Lipschitz constant equal $1$. Observe that 
\[
\{f=y\} \subset (\Omega\setminus Z(f)) \cup (\pi^{-1} (\{y\}) \cap \Gamma (f))\, .
\]  
It suffices therefore to show that 
\begin{equation}\label{e:intermedia}
\mathcal{H}^{n-s} (\Gamma (f) \cap \pi^{-1} (\{y\})) = 0 \qquad \mbox{for a.e. $y$.}
\end{equation}
However, we can use Federer's coarea inequality for Lipschitz maps (see \cite[Theorem 2.10.25]{Federer}) to bound
\begin{equation}\label{e:coarea-ineq-Fed}
\int_{\mathbb R} \mathcal{H}^{n-s} (\Gamma (f) \cap \pi^{-1} (\{y\}))\, dy 
\leq {\rm Lip}\, (\pi) \mathcal{H}^{n+1-s} (\Gamma (f))\, .
\end{equation}
Therefore \eqref{e:intermedia} follows immediately from Theorem \ref{t:graph}.
\end{proof}

The next sections will be dedicated to prove Theorem \ref{t:graph}. In the remaining part of this section we will recall some known facts about Sobolev-Slobodeckij and Besov spaces and the corresponding capacities.

\subsection{Sobolev-Slobodeckij and Besov spaces} We start by recalling 
the definition of the Sobolev-Slobodeckij spaces.

\begin{definition}\label{d:Sob-Slob}
Consider $s\in(0,1)$, $p\in [1,\infty)$ and an open set $\Omega\subset \mathbb R^n$. $W^{s,p} (\Omega)$ is the subset of $L^p (\Omega)$ for which the following seminorm is finite:
\[
\|f\|_{\dot{W}^{s,p}(\Omega)}=\left(\int_\Omega\int_\Omega\frac{|f(y)-f(z)|^p}{|y-z|^{sp+n}}\,dy\,dz\right)^{1/p}.
\]
On $\mathbb R^n$ it is also useful to introduce the space
$\dot{W}^{s,p}(\R^n)$, which is defined to be the closure of $C^{\infty}_c(\R^n)$ under the $\dot{W}^{s,p} (\mathbb R^n)$ norm.
\end{definition}

\begin{remark}\label{r:extension}
We recall that, if $\Omega$ is a bounded open set with sufficiently regular boundary, then $f\in W^{s,p} (\Omega)$ if and only $f\in L^p (\Omega)$ and has an extension $\tilde{f}\in \dot{W}^{s,p} (\mathbb R^n)$. 
\end{remark}

Moreover, it is convenient to introduce the functions
\begin{align}
D^{s,p}f(x) &:=\left(\int_{\R^n}\frac{|f(x+h)-f(x)|^p}{|h|^{sp+n}}\,dh\right)^{1/p},\label{eq:def of fractional deriavitve}\\
D^{s,p}_rf(x)&:=\left(\int_{|h|\le r}\frac{|f(x+h)-f(x)|^p}{|h|^{sp+n}}\,dh\right)^{1/p}\, .\label{eq:def of local fractional deriavitve}
\end{align}
The motivation for introducing the first function is that $\|f\|_{\dot{W}^{s,p} (\mathbb R^n)} = \|D^{s,p} f\|_{L^p (\mathbb R^n)}$, while clearly the second one is a ``local'' version of the first which is useful when the domain of $f$ is not the entire space. 

As $W^{s,p} (\mathbb R^n) =B^s_{p,p} (\mathbb R^n)$ is part of the Besov spaces family, we will also consider some statements for the Besov spaces $B^s_{pq}$ with $s\in \mathbb R$, $1\le p<\infty$, and $1\le q\le \infty$.

\begin{definition}
The Besov norms $B^s_{pq} (\mathbb R^n)$ and $\dot{B}^s_{pq} (\mathbb R^n)$ are defined as
\begin{align*}
\|f\|_{B^{s}_{pq} (\mathbb R^n)}&:=\left(\|P_{\le -1}f\|^q+\sum_{k=0}^\infty \|2^{sk}P_kf\|_p^q\right)^{1/q},\\
\|f\|_{\dot{B}^s_{pq} (\mathbb R^n)}&:=\left(\sum_{k=-\infty}^\infty \|2^{sk}P_kf\|_p^q\right)^{1/q},
\end{align*}
where $P_k$, $k\in \mathbb Z$, are the usual dyadic Littlewood-Paley projections, and the case $q=\infty$ is interpreted as an $l^\infty$ norm in $k$.
\end{definition}

\begin{remark}\label{r:duality}
In our considerations we will use the duality pairing inequality for $B^s_{pq}$ and 
$B^{-s}_{p^\prime q^\prime}$, and the corresponding ones for $\dot{B}^s_{pq}$ and $\dot{B}^{-s}_{p^\prime q^\prime}$ (as usual $p^\prime$ and $q^\prime$ are the dual exponents namely $\frac{1}{p}+\frac{1}{p^\prime}=1 = \frac{1}{q} +\frac{1}{q^\prime}$). In particular, if both $f$ and $g$ are Schwartz functions, we have the inequality
\begin{equation}\label{e:duality}
\int fg \leq \|f\|_{B^s_{pq}} \|g\|_{B^{-s}_{p^\prime q^\prime}}, 
\end{equation}
(and the corresponding one with $\dot{B}^s_{pq}$ and $\dot{B}^{-s}_{p^\prime q^\prime}$).
We do not delve into the functional analytic tools and the theory of distributions needed to make sense of \eqref{e:duality} when $f$ and $g$ are less regular. We will however need \eqref{e:duality} when $f$ is taken to be continuous and $g$ a Borel measure. The inequality can then be easily justified through a standard approximation procedure.
\end{remark}

When $0<s<1$, a set of equivalent norms for $\dot{B}^s_{pq}$ (which by abuse of notation we keep denoting as $\|\cdot\|_{\dot{B}^s_{pq} (\mathbb R^n)}$) is given by
\[
\|f\|_{\dot{B}^s_{pq}}=\left(\int_{\R^n}\left(\int_{\R^n}|f(x+h)-f(x)|^p\,dx\right)^{q/p}\,\frac{dh}{|h|^{sq+n}}\right)^{1/q},
\]
again with the proper interpretation when $q=\infty$.
Clearly the above  $\dot{B}^s_{pp}$-norm is exactly the $\dot{W}^{s,p}$-norm.

\subsection{Besov capacity} We next introduce the Besov capacity. 

\begin{definition} Let $s>0$, $1\le p <\infty$, and $1\leq q \leq \infty$. Then, given a bounded set $E\subset \mathbb R^n$, 
\begin{equation}\label{e:capacity}
\Capa_{spq} (E) =\inf \left\{ \|f\|^p_{B^s_{p,q}}: f\in B^s_{pq} (\mathbb R^n) \quad\mbox{s.t.} \quad E\subset\{f\ge1\}^\circ\right\}\, ,
\end{equation}
where $\Omega^\circ$ denotes the interior of a set $\Omega$. 
\end{definition}

For the following theorem we refer to \cite{Dor87} for a proof and \cite{Net89} for another proof that contains also the quasi-Banach cases.

\begin{theorem}\label{t:capacity}
Let $s>0$, $1\le p<\infty$, and $1\le q\le \infty$. Let $f\in B^s_{p,q}(\R^n)$ and let $Z(f)$ be the set of Lebesgue points of $f$.
Then $\Capa_{s,p,q} (\R^n\setminus Z(f)) = 0$. 
\end{theorem}

\begin{remark}\label{r:capacity-Hausdorff}
Theorem \ref{t:capacity} implies immediately that 
\begin{equation}\label{e:good-points}
\mathcal{H}^{n-s} (\Omega \setminus Z(f)) =0 \qquad \mbox{for every $f\in W^{s,p} (\Omega)$ with $0<s<1$ and $1\leq p \leq \infty$}.
\end{equation}
Since this can be reduced to a local statement, we can assume that $\Omega$ is bounded and smooth and, therefore (thanks to the extension theorem, cf. Remark \ref{r:extension}), that it is the whole space $\mathbb R^n$. Next recall that:
\begin{itemize}
\item if $p>1$, then $\Capa_{s,p,q} (E) =0$ implies ${\rm dim}_H (E) \leq n -sp$;
\item $\Capa_{s,1,1} (E) =0$ implies $\mathcal{H}^{n-s} (E) =0$.
\end{itemize}
These two facts are indeed (modulo technicalities) a consequence of Theorem \ref{t:capacity-estimate} below.
Given that $W^{s,p} (\mathbb R^n) = B^s_{pp} (\mathbb R^n)$, we immediately conclude \eqref{e:good-points}.
\end{remark}

\section{Capacity estimate}

The following result in potential theory is well known, see \cite{Ada89} for an excellent exposition. Since our argument relies heavily on this estimate, we will present the proof. 

\begin{theorem}\label{t:capacity-estimate}
Let $s\ge0$, $1\le p<\infty$, $1\le q \le \infty$, and $m\ge 0$. Let $f\in B^{s}_{p,q}(\R^n)$ with $f\ge0$ and $\liminf_{r\downarrow0}f_{x,r}\ge 1$ on a Borel set $E$. Then there is a geometric constant $C=C(s,p,q,n,m)$ such that:
\begin{itemize}
\item[(a)] If $m>n-sp$,
\begin{equation}\label{e:capacity-estimate-1}
\H^m_\infty(E)\le C\|f\|^p_{B^s_{p,q}}\, ;
\end{equation}
\item[(b)] If $q=1$,
\begin{equation}\label{e:capacity-estimate-2}
\H^{n-s}_{\infty}(E)\le C\|f\|^p_{\dot{B}^s_{p,1}}\, .
\end{equation}
\end{itemize}
\end{theorem}

The theorem has the following simple corollary.

\begin{corollary}\label{c:capacity-estimate}
If $f\in W^{s,p}(\R^n)$, $s\in(0,1)$, $1\le p<\infty$, $\liminf_{r\downarrow0}f_{x,r}\ge 1$ on $E$, then
\begin{equation}\label{e:capacity-estimate-3}
\H^{n-s}_\infty(E)\le C\|f\|_{\dot{W}^{s,p}(\R^n)}\cdot\|f\|_{L^p(\R^n)}^{p-1}\, .
\end{equation}
Moreover, if in addition $r\leq 1$, $E\subset B_r$ and $\spt f\subset B_{2r}$, then
\begin{equation}\label{e:capacity-estimate-4}
\H^{n-s}_\infty(E)\le Cr^{sp-s}\|f\|_{\dot{W}^{s,p}(B_{3r})}^p\, .
\end{equation}
\end{corollary}
We will prove the Corollary below. For now, let us start with the proof of the Theorem.
\begin{proof}[Proof of Theorem \ref{t:capacity-estimate}]
We are going to use the famous Frostman's lemma, see \cite{Fro35} (or a standard textbook like \cite{Mat95}): there exists a Borel measure $\mu$ supported in $E$ (namely $\mu (E^c)=0$), which satisfies the bound $\mu(B_r (x))\le r^m$ for all balls and such that $\H^{m}_\infty(E)\le C\mu(E)$, where $C$ is a geometric constant. The problem reduces, therefore, to bounding the total mass of the Frostman measure $\mu$. Let us first assume that $f\in C^0\cap B^{s}_{p,q}$ and $\mu(E)<\infty$, and consider that
\[
\mu(E)\le\int f\,d\mu\le C\|f\|_{B^s_{p,q}}\|\mu\|_{B^{-s}_{p^\prime,q^\prime}}\, ,
\]
where $p^\prime$ and $q^\prime$ are the exponents dual to $p$ and $q$. 

Recalling now the Paley-Littlewood decomposition, let $\phi$ be the $C^\infty_c$ function supported in the annalus $\frac{1}{2} \leq |\xi|\leq 2$ with the property that 
\[
\widehat{P_k f} (\xi) = \hat{f} \phi (2^{-k} \xi)\, .
\]
Let $m_0$ be the inverse Fourier transform of $\phi$ and recall that 
\[
P_k \mu(x) = \int_{\mathbb{R}^n} 2^{nk} m_0(2^k(x-y)) \, d\mu(y)\, .
\]
Finally recall that, for every positive $a$ there is a constant $C$ such that $|m_0 (z)|\leq C (1+|z|)^{-a}$

Now estimate
\begin{align*}
|P_k\mu|(x) \le &\int_{|x-y|\le 2^{-k}}2^{nk}m_0(2^k(x-y))\,d\mu(y)+\sum_{\ell< k}\int_{2^{-(\ell+1)}<|x-y|\le 2^{-\ell}}2^{nk}m_0(2^k(x-y))\,d\mu(y)\\
\lesssim & 2^{nk}2^{-mk}+\sum_{\ell<k}2^{nk}2^{a(\ell-k)}2^{-m\ell}\, ,
\end{align*}
where from now on we use the symbol $\lesssim$ in case the left hand side can be bounded by the right hand side up to a geometric constant (depending only on $n,s,p,q$, and $m$.

Choose then an $a>m$ so that the latter series converges to obtain the bound
\[
\|P_k\mu\|_{L^\infty}\lesssim 2^{(n-m)k}\, .
\]
On the other hand,
\[\|P_k\mu\|_{L^1}\le \int\int 2^{nk}|m_0|(2^{k}(x-y))\,dx\,d\mu(y)=C\mu(E).\]
Notice $p^\prime>1$. We can thus bound
\[\|2^{-sk}P_k\mu\|_{p^\prime}\le 2^{-sk}\|P_k\mu\|_{\infty}^{1-1/p^\prime}\|P_k\mu\|_{1}^{1/p^\prime}\lesssim 2^{\frac{(n-sp-m)}{p}k}\mu(E)^{1/p^\prime}.\]
Let $c_k=2^{\frac{(n-sp-m)}{p}k}$. Then,
\[\|\mu\|_{B^{-s}_{p^\prime,q^\prime}}\lesssim \mu(E)^{1/p^\prime}\cdot \left\|(c_k)_{k\ge -1}\right\|_{\ell^{q^\prime}}\lesssim\mu(E)^{1/p^\prime},\]
provided $m>n-sp$, where the $\ell^{q^\prime}$ norm is taken over the sequence indexed by $k\ge -1$. Combining these we get $\H^m_\infty(E)\lesssim\mu(E)\lesssim \|f\|_{B^{s}_{p,q}}^p$.
As for the estimate in (b), notice that, if $q=1$, then $q^\prime=\infty$ and we therefore get
\[
\|\mu\|_{\dot{B}^{-s}_{p^\prime,\infty}}\lesssim \mu(E)^{1/p^\prime}\cdot\sup_{k\in\Z}\left(2^{\frac{(n-sp-m)}{p}k}\right)\, .
\]
Then the left hand side remains bounded even in the critical case $m=n-sp$.

In order to remove the assumptions that $\mu$ is bounded and $f$ continuous, consider first a standard nonnegative mollification kernel $\varphi$ and denote by $f_\ell$ the mollification of $f$ with $\varphi_\ell (x):= \ell^{-n} \varphi (\ell^{-1} x)$. Next let $B_R$ be the ball centered at $0$ with radius $\ell^{-1}$. Then we can apply the above argument to conclude
\[
\int_{B_R} f_{\ell}\,d\mu\le C\|f_{\ell}\|_{B^s_{p,q}}\mu(E\cap B_R)^{1/p^\prime}\, .
\]
We next let $\ell \downarrow 0$ and use Fatou's lemma and the inequality $\|f_{\ell}\|_{B^{s}_{p,q}}\le\|f\|_{B^{s}_{p,q}}$ to get
\[
\mu (E\cap B_R) \leq \int_{B_R} \liminf_{\ell \downarrow 0} f_{\ell}\, d\mu \leq C \|f\|_{B^s_{pq}}\mu (E\cap B_R)^{1/p^\prime}\, .
\]
In particular, since $\mu (E\cap B_R)$ is finite, we conclude
\[
\mu (E\cap B_R) \leq C \|f_\ell\|_{B^s_{pq}}^p\, .
\]
Hence letting $R\uparrow \infty$ we achieve both the estimates \eqref{e:capacity-estimate-1} and \eqref{e:capacity-estimate-2} in full generality.
\end{proof}

\begin{proof}[Proof of Corollary \ref{c:capacity-estimate}]
For $p=1$ the first inequality is the statement of Theorem \ref{t:capacity-estimate}(b). For $p>1$ we choose $m=n-s>n-sp$ and ``homogeneize'' the norms with the classical rescaling trick. More precisely, we let 
\[
\lambda E := \{\lambda x: x\in E\}
\]
and $f^\lambda (x) := f (\lambda^{-1} x)$. Clearly $\liminf_{r\downarrow 0} f^\lambda_{x,r} \geq 1$ for all $x\in \lambda E$. Thus we can estimate
\[
\mathcal{H}_\infty^{n-s} (\lambda E) \leq 
C (\|f^\lambda\|_{L^p}^p + \|f^\lambda\|_{\dot{W}^{s,p}}^p)\, 
\]
using Theorem \ref{t:capacity-estimate}(a).
Note the obvious scaling facts that 
\begin{align*}
\mathcal{H}_\infty^{n-s} (\lambda E) &= \lambda^{n-s}\mathcal{H}_\infty^{n-s} (E),\\
\|f^\lambda\|^p_{L^p} & = \lambda^n \|f\|^p_{L^p},\\
\|f^\lambda\|^p_{\dot{W}^{s,p}}&= \lambda^{n-ps}\|f^\lambda\|^p_{\dot{W}^{s,p}}\, .
\end{align*}
In particular we get 
\[
\mathcal{H}_\infty^{n-s} (\lambda E) \leq C \bigl(\lambda^s\|f\|^p_{L^p}
+ \lambda^{s-ps} \|f\|_{\dot{W}^{s,p}}^p)\, .
\]
Choosing 
\[
\lambda^s= \frac{\|f\|_{\dot{W}^{s,p}}}{\|f\|_{L^p}}\, ,
\]
we then get \eqref{e:capacity-estimate-3}. 

For the second part, first observe that \eqref{e:capacity-estimate-3} yields immediately
\begin{align*}
\mathcal{H}^\infty_{n-s} (E) \leq C \|f\|_{L^p (B_{2r})}^{p-1} \|f\|_{\dot{W}^{s,p} (\mathbb R^n)}\, .
\end{align*}
Moreover, since support of $f$ is in $B_{2r}$, the Poincar\'e inequality for fractional Sobolev spaces gives 
\[
\|f\|_{L^p (B_{2r})} \leq C r^s \|f\|_{\dot{W}^{s,p} (B_{3r})}\, .
\]
Next we estimate
\begin{align*}
\|f\|_{\dot{W}^{s,p} (\mathbb R^n)}^p
&= \int_{\mathbb R^n}\int_{\mathbb R^n} \frac{|f(y)-f(z)|^p}{|y-z|^{sp+n}}\,dy\,dz\\
&= \int_{B_{3r}} \int_{B_{3r}} \frac{|f(y)-f(z)|^p}{|y-z|^{sp+n}}\, dy\,dz + 2 \int_{\mathbb R^n\setminus B_{3r}} \int_{B_{2r}} \frac{|f(y)|^p}{|y-z|^{sp+n}}\, dy\,dz\\
& \leq \|f\|_{\dot{W}^{s,p} (B_{3r})} + 2 \int_{B_{2r}} |f(y)|^p dy \int_{\mathbb R^n \setminus B_r} \frac{dw}{|w|^{sp+n}}\\
& \leq \|f\|_{\dot{W}^{s,p} (B_{3r})} + \frac{C}{r^{sp}} \|f\|_{L^p (B_{2r})}^p 
\leq C \|f\|_{\dot{W}^{s,p} (B_{3r})}^p\, ,
\end{align*}
thus completing the proof.
\end{proof}

\section{Proof of Theorem \ref{t:coarea-fractional}}

The proof of the theorem follows closely the scheme of \cite{MSZ03}. First of all we will show the following lemma. We will use routinely the orthogonal projection of $\mathbb R^n \times \mathbb R$, for which we introduce the notation $\pi_o$. Recall the definitions \eqref{eq:def of fractional deriavitve} and \eqref{eq:def of local fractional deriavitve}.

\begin{lemma}\label{l:cylinder-estimate}
Let $0<s<1$, $1\leq p < \infty$ and $n\geq 1$. Given a point $z_0 = (x_0, y_0) \in \mathbb R^n\times \mathbb R$ and a radius $r>0$, let $Q(z_0,r)\subset \mathbb R^{n+1}$ be the cylinder $B(x_0,r)\times B(y_0,r^s)$. 
Then, for every $f\in \dot{W}^{s,p} (B_{5r} (x_0))$ we have the estimate
\[
\H^{n-s}_{\infty}(\pi_o (\Gamma (f)\cap Q(z_0,r)))\le Cr^{-s}\int_{\pi_o (\Gamma (f) \cap Q(z_0,2r))} (|D^{s,p}_{6r} f|^p+1)\, ,
\]
where the constant $C$ depends on $n$, $s$, and $p$, but not on $f$ nor on $r$.
\end{lemma}

\begin{proof}
Set $E:=\pi_o (\Gamma (f)\cap Q(z_0,r))$, and notice that $x\in E$ if and only if $|x-x_0|<r$, $x\in Z(f)$, and $|\bar{f}(x)-y_0|<r$. Now fix two smooth cut-off functions $\varphi, \psi$ on $\mathbb R^n$ and $\mathbb R$ respectively, with $\mathbf{1}_{B_1}\le\varphi\le \mathbf{1}_{B_2}$ and $\mathbf{1}_{[-1,1]} \le \psi \le\mathbf{1}_{[-2^s, 2^s]}$. We define 
\begin{align*}
\varphi_{x_0 r}(x) &:= \varphi \left(\frac{x-x_0}{r}\right),\\
\psi_{y_0,r} (y) &:= \psi \left(\frac{\bar{f}(x)-y_0}{r^s}\right),
\end{align*}
and 
\[
u(x) =\varphi_{x_0, r} (x)\, \psi_{y_0,r} (\bar{f} (x))\, .
\]
Observe now that:
\begin{itemize}
\item $E\subset B_r (x_0)$;
\item ${\rm spt}\, (u) \subset B_{2r} (x_0)$;
\item $\liminf_{r\downarrow 0} u_{x,r}\geq 1$ on every $x\in E$.
\end{itemize}
In particular we can apply Corollary \ref{c:capacity-estimate} and use \eqref{e:capacity-estimate-4} to conclude
\[
\H^{n-s}_\infty(E)\le Cr^{sp-s}\|u\|^p_{\dot{W}^{s,p}(B_{3r} (x_0))}\, .
\]
Next, observe that 
\begin{align*}
\|u\|^p_{\dot{W}^{s,p}(B_{3r}(x_0))} 
&= 2 \int_{B_{3r} (x_0) \cap \{u\neq 0\}} \int_{B_{3r} (x_0)} \frac{|u(x)-u(y)|^p}{|x-y|^{n+sp}}\, dx\, dy \\
&\leq 2 \int_{B_{3r} (x_0)\cap \{u\neq 0\}} |D^{s,p}_{6r} u (y)|^p\, dy 
\leq 2 \int_{\pi_o (\Gamma (f)\cap Q_{2r} (z_0))} |D^{s,p}_{6r} u|^p\, . 
\end{align*}
We next estimate
\begin{align*}
D^{s,p}_{6r}u(x) & \le \|\varphi_{x_0,r}\|_{L^\infty}D ^{s,p}_{6r} (\psi_{y_0,r}(f (x)))+\|\psi_{y_0,r}\|_{L^\infty} D^{s,p}_{6r}(\varphi_{x_0,r}(x))\\
& \le D^{s,p}_{6r}(\psi_{y_0,r}(f(x)))+D^{s,p}_{6r}(\varphi_{x_0,r}(x))\, .
\end{align*}
Observe that the Lipschitz constant of $\psi_{y_0,r}$ is bounded by $C r^{-s}$, while 
\[
D^{s,p}_{6r} \varphi_{x_0,r} \leq D^{s,p} \varphi_{x_0,r} \leq C r^{-s}\, .
\] 
Therefore the conclusion of the lemma readily follows.
\end{proof}

\begin{proof}[Proof of Theorem \ref{t:coarea-fractional}] Consider $f$ as in the statement of the theorem, and the corresponding exponents $0<s<1$ and $1\leq p < \infty$.
We fix $r_0>0$ and smaller than $1$ and introduce the measure 
\begin{equation}\label{e:sigma-measure}
\sigma(E) : =\int_{\pi_o (\Gamma (f) \cap E)}\left(|D^{s,p}_{6r_0} f|^p+1\right)
\end{equation}
which we define for every Borel set $E$ which is contained in $\Omega'\times \mathbb R$, where $\Omega'$ is the subset of $\Omega$ consisting of those points $x$ of $\Omega$ at distance at least $6r_0$ from the boundary $\partial \Omega$. Provided we show that $\pi_o (\Gamma (f) \cap E)$ is Lebesgue measurable (a routine check, the reader can argue as in \cite{MSZ03}), $\sigma$ defines a Borel measure. 

Consider now an arbitrary point $z= (x,y) \in \Omega'\times \mathbb R$ and  a cylinder $Q_r (z) = B_r (x) \times (y_0-r^s, y_0+r^s)$ for some arbitrary positive radius $r< r_0$.  Observe that 
\[
\Gamma (f) \cap Q_r (z) \subset \bigl(\pi_o (\Gamma (f)\cap Q_r (x))\bigr) \times (y_0-r^s, y_0+r^s)\, .
\]
We recall the elementary inequality 
\begin{equation}\label{e:Hausdorff-product}
\mathcal{H}^{c+1}_\infty (A \times (a,b))\leq C (b-a) \mathcal{H}^c_\infty (A)\, , 
\end{equation}
which is valid for an arbitrary interval $(a,b)\subset \mathbb R$, an arbitrary exponent $c$, and a set $A\subset \mathbb R^n$ with $\operatorname{diam} A\le b-a$ (with a constant $C$ depending only upon $n$ and $c$). To prove it, consider a cover for $A$ with sets $E_i$, so that 
\[
\omega_c \sum_{i=0}^\infty ({\rm diam}\, (E_i))^c \leq 2 \mathcal{H}^c_\infty (A) \, .
\]
We may assume $\operatorname{diam}E_i\le b-a$ after intersecting the set with $A$. Now for
each $i$, split the interval $(a,b)$ in $\lceil(b-a)/(\operatorname{diam} (E_i))\rceil$
intervals of length $\le\operatorname{diam} (E_i)$ and take all their products with
$E_i$: this makes a family $\mathcal{F}_i = \{F_{i,j}\}_j$ of sets, and their union $\mathcal{F}= \bigcup_i \mathcal{GF}_i$ is a cover of $A\times (a,b)$. Moreover, clearly 
\[
{\rm diam}\, (F_{i,j}) \leq \sqrt{2} {\rm diam}\, (E_i)\, .
\]
In particular
\begin{align*}
\mathcal{H}^{c+1}_\infty (A\times (a,b)) &\leq \omega_{c+1} \sum_{i=0}^\infty \sum_{j=0}^\infty \bigl({\rm diam}\, (F_{i,j})\bigr)^{c+1}\\
& \leq \omega_{c+1} \sum_{i=0}^\infty 2^{(c+1)/2} \frac{2 (b-a)}{{\rm diam}\, (E_i)}\bigl( {\rm diam}\, (E_i)\bigr)^{c+1} \leq \frac{2^{(c+1)/3} \omega_{c+1}}{\omega_c} \mathcal{H}^c_\infty (A)\, .
\end{align*}
Having shown \eqref{e:Hausdorff-product} we can estimate 
\begin{equation}\label{e:local}
\H^{n-s+1}_\infty(\Gamma (f) \cap Q_r (z))\le Cr^s\H^{n-s}_\infty(\pi_o(\Gamma (f) \cap Q_r (z)))\le C\sigma(Q_{2r} (z))\, ,
\end{equation}
where in the last inequality we have used Lemma \ref{l:cylinder-estimate}. 

We next introduce the infinite cylinders $C_r (x) := B_r (x) \times \mathbb R \subset \mathbb R^{n+1}$ and use \eqref{e:local} to prove that 
\begin{equation}\label{e:global}
\H^{n-s+1}_\infty (\Gamma (f) \cap C_\rho (x)) \leq C \sigma (C_{4\rho} (x))
\qquad \forall x\in \Omega', \forall \rho < \tfrac{1}{6} {\rm dist}\, (x, \partial \Omega)\, .
\end{equation}
To see this, fix a very large integer $k\in \mathbb N$ and cover $B_\rho (x)$ with balls $B_{r}^i$ of radius $r = 2^{-k} \rho$, centered at points of $B_\rho (x)$ and with the finite overlap property: each enlarged ball $B_{2\rho}(x)$ in the collection intersects at most $N = N (n)$ other enlarged balls of the collection. Consider then any cylinder $Q_{r}^{ij} = B_{r}^i \times ((j-1) r^{s}, (j+1) r^{s})$ Clearly this collection of cylinders cover $C_\rho (x)$ and each enlarged cylinder $Q_{2r}^{ij}$ in the collection intersects at most $5N+4$ other cylinders in the collection. Moreover, if $r$ is picked small, the cylinders $Q_{2r}^{ij}$ are contained in $C_{4\rho} (x)$. Using the subaddivity of the Hausdorff premeasure $\mathcal{H}^{n-s+1}_\infty$, \eqref{e:local}, the finite overlap property, and the additivity of the measure $\sigma$ we immediately get
\[
\H^{n-s+1}_\infty (\Gamma (f) \cap C_\rho (x)) 
\leq \sum_{i,j} \H^{n-s+1}_\infty (\Gamma (f) \cap Q_r^{ij})
\leq C \sum_{i,j} \sigma (Q_{2r}^{ij}) \leq C \sigma (C_{4\rho} (x))\, .
\]
Observe next that the measure $\sigma$ depends in fact upon $r_0>0$, cf. \eqref{e:sigma-measure}.  On the other hand \eqref{e:global} has been proved for any $\rho$, independently of $r_0$. We can thus fix $x$ and $\rho$ as in \eqref{e:global} and let $r_0$ go to $0$. Recalling however the definition of $D^{s,p}_{r_0} f$, we immediately see that 
\[
\lim_{r_0\downarrow 0} D^{s,p}_{r_0} f (x) = 0 \qquad \mbox{for a.e. $x$,}
\]
from the finiteness of the Sobolev-Slobodeckij norm. 
We can pick an extension $g\in \dot{W}^{s,p} (\mathbb R^n)$ of $f$ and get that $D^{s,p}_{r_0} f (x) \leq D^{s,p} g (x)$ and thus use the dominated convergence theorem to see that 
\[
\lim_{r_0\downarrow 0} \sigma (E) = |\pi_o (E\cap \Gamma (f))|
\leq |\pi_o (E)|\, 
\]
(where $|A|$ denotes the Lebesgue measure of $A$ in $\mathbb R^n$). We thus get the estimate
\begin{equation}\label{e:global-2}
\H^{n-s+1}_\infty (\Gamma (f) \cap C_\rho (x)) \leq C \rho^n\, ,
\end{equation}
where $\rho$ is a geometric constant depending only upon $s,p$, and $n$. This independence of the constant upon $f$ lets us conclude. In fact, if we fix any $\ell>0$ and consider the function $f_\ell (x) = \ell f (\ell^{-1} x)$, then clearly $f_\ell \in W^{s,p} (\ell \Omega)$ and we can apply \eqref{e:global-2} to $f_\ell$. On the other hand $\Gamma (f_\ell) = \ell \Gamma (f)$ and obvious scaling arguments yield
\begin{align}
\ell^{n-s+1} \H^{n-s+1} (\Gamma (f) \cap C_\rho (x)) 
& = \H^{n-s+1} (\ell (\Gamma (f) \cap C_\rho (x)))\nonumber\\
&= \H^{n-s+1} (\Gamma (f_\ell) \cap C_{\ell \rho} (\ell x) 
\leq C \ell^n \rho^n\, .\label{e:global-3}
\end{align} 
In particular 
\[
\H^{n-s+1} (\Gamma (f) \cap C_\rho (x)) \leq C \ell^{s-1} \rho^n\, .
\]
Letting $\ell\uparrow \infty$ we achieve 
\[
\H^{n-s+1} (\Gamma (f)\cap C_\rho (x)) = 0\, .
\]
Covering $\Gamma (f)$ with countably many such cylinders concludes the proof.
\end{proof}

\section{Intermittent random series of wavelets}\label{s:random}

In this section we consider some possible constructions of random functions in the Sobolev-Slobodeckij hierarchy. For simplicity we fix our domain to be the periodic torus $\mathbb T^n$. We then fix an orthogonal basis of wavelets $\{\psi_{k, \lambda}\}$ on the space of mean-zero elements of $L^2 (\mathbb T^n)$, where $k\in \mathbb N$ denotes the scale of the wavelet, $\Lambda_k$ denotes the subset of wavelets (in the basis) of scale $k$, and the index $\lambda$ runs in $\Lambda_k$. We will assume that this basis is $L^\infty$ normalized, that it is constructed as the standard periodization of a basis of compactly-supported wavelets in $\mathbb R^n$  (in Appendix \ref{a:multi-resolution} we give more details about this to the interested reader) and that the mother wavelets of the basis are $C^1$. This allows us to invoke the following (cf. \cite[Chapter 9]{Dau92}).

\begin{theorem}\label{t:PL-wavelets}
Let $0<\alpha <1$ and $1 < p < \infty$. A function $f\in L^2 (\mathbb T^n)$ with expansion
\[
f = \sum_k \sum_{\lambda \in \Lambda_k} c_{k,\lambda} \psi_{k,\lambda}
\]
belongs to $W^{\alpha,p}$ if and only if 
\[
\|f\|^p_{\alpha,p} := \sum_k 2^{k (p \alpha -n)} \sum_{\lambda \in \Lambda_k} |c_{k,\lambda}|^p < \infty\, .
\]
Moreover there is $C= C(n,\alpha, p)>0$ such that $C^{-1} \|f\|_{\alpha, p} \leq \|f\|_{W^{\alpha,p}} \leq C \|f\|_{\alpha, p}$ for all $f$.
\end{theorem}

We now fix $0<s <1$ and $1<p<\infty$ and choose a number $\beta$ satisfying
\begin{align}
\beta & > - s\, ,  \label{e:constraint-1}\\
n & > p \left(s+\beta\right) =: \delta\, \label{e:constraint-2}
\end{align}
(the existence of $\beta$ is obvious).
We then build a series of wavelets
\begin{equation}\label{e:random-wavelets}
f = \sum_{k\in \mathbb N} \sum_{\lambda \in \Lambda_k} c_{k, \lambda} \psi_{k, \lambda} \, ,
\end{equation}
with random coefficient $c_{k,\lambda}$ given by
\begin{equation}\label{e:random-coefficient}
c_{k, \lambda} = W_{k, \lambda} \left( 2^{-ks} (1- B_{k, \lambda}) + 2^{k\beta} B_{k, \lambda}\right)\, ,
\end{equation}
where 
\begin{itemize}
\item[(a)] the $W_{k, \lambda}$ are independent identically distributed real random variables with mean zero, variance $1$, and finite moments;
\item[(b)] the $B_{k, \lambda}$ are independent Bernoulli random variables taking the value $1$ with probability $2^{-\delta k}$ and the value $0$ with probability $1-2^{-\delta k}$.
\end{itemize}
The probability measure constructed above (for a fixed choice of $s, p$, and $\beta$ satisfying 
the constraints) will be denoted by $\mathbb P$ and, since for every choice of the parameters we have $f\in L^1 (\mathbb T^n)$ almost surely, we will consider $\mathbb P$ as a probability on $L^1 (\mathbb T^n)$.

The following  theorem justifies the term ``intermettent random series of wavelets''. 

\begin{theorem}\label{t:random-Sobolev}
Let $0< s < 1$, $1<p<\infty$ and $\beta$ satisfying \eqref{e:constraint-1} and \eqref{e:constraint-2}. Consider random $f$ as in \eqref{e:random-wavelets}, satisfying \eqref{e:random-coefficient}, (a), and (b) above. 
Then the following holds for $\mathbb P$-a.e. $f$:
\begin{itemize}
\item[(i)] $f\in W^{\alpha,p}$ for every $\alpha< s$;
\item[(ii)] $f\not\in W^{\alpha,p}$ for any $\alpha > s$;
\item[(iii)] For every $q>p$ there is $\alpha < s$ such that $f\not\in W^{\sigma, q}$ for all $\sigma \in (\alpha, s)$.
\end{itemize}
\end{theorem}

The proof is a straightforward application of elementary facts in probability, but in any case we include it in Appendix \ref{a:random}.
The aim of this second part of the paper is to generalize Kahane's approach in \cite{Kahane} to study the dimension of the level set of a random function as in Theorem \ref{t:random-Sobolev}. In particular we will prove the following result. 

\begin{theorem}\label{t:random-level-sets}
Consider $s,p, \beta$, and $\mathbb P$ as in Theorem \ref{t:random-Sobolev}. Assume that the basis of wavelets satisfies the following separation property for some fixed constant $C>0$:
\begin{itemize}
    \item[(S)] For every $x,x'\in \mathbb T^n$ {\em distinct} there is $k\in \mathbb N$ and $\lambda \in \Lambda_k$ such that 
    \begin{equation}\label{e:separation}
    C^{-1} 2^{-k} \leq d (x,x') \leq C 2^{-k}\, , \quad |\psi_{k, \lambda} (x)|\geq C^{-1}\, , \quad \mbox{and} \quad \psi_{k,\lambda} (x')=0\, .
    \end{equation}
\end{itemize}
Assume moreover that the random variables $W_{k,\lambda}$ are all Gaussians. Then
\begin{itemize}
\item[(i)] For $\mathbb P$-a.e. $f$, there is a set $Y(f)\subset \mathbb R$ of positive Lebesgue measure such that 
\[
{\rm dim}_H (Z(f) \cap \{y=f\}) = n-s \qquad \mbox{for every $y\in Y(f)$;}
\]
\item[(ii)] For a.e. $y$ there is a set $F (y)\subset L^1 (\mathbb T^n)$ with $\mathbb P (F(y))>0$ such that 
\[
{\rm dim}_H (Z(f)\cap \{y=f\}) = n-s \qquad \mbox{for every $f\in F(y)$.}
\]
\end{itemize}
\end{theorem}

Recall here that $Z(f)$ denotes the set of Lebesgue points of $f$. See Definition~\ref{d:precise-representative}. The assumption (S) is very mild and in particular it is satisfied by a number of wavelets basis, we refer to Appendix \ref{a:multi-resolution} for a discussion of this point, see in particular Theorem \ref{t:separation}. Note moreover that, by the {\em deterministic} coarea inequality, namely Theorem \ref{t:coarea-fractional}, in both (i) and (ii) it suffices in fact to prove the lower bound ${\rm dim}_H (Z(f)\cap \{y=f\}) \geq n-s$. In particular, the part of the monograph \cite{Kahane} which is relevant to us is the one in which Kahane derives lower bounds on the Hausdorff dimension of the level sets. 

\smallskip

Kahane's approach is based on the construction of a suitable measure $\mu^y$ ``concentrated'' on $\{f=y\}$ (namely such that $\mu^y (\{f\neq y\})=0$). For the probabilities $\mathbb P$ considered in  \cite{Kahane}, $f$ is continuous $\mathbb P$-almost surely, and thus $\{f=y\}$ is a closed set. Since the measure $\mu^y$ is constructed through an approximation procedure as the weak$^*$ limit of suitable absolutely continuous approximations supported in the sets $\{y-\varepsilon \leq f \leq 1+\varepsilon\}$, in Kahane's case the fact that $\mu (\{f\neq y\})=0$ is an easy consequence of weak$^*$ convergence. In our case, for a large subset of the space of our paramaters $s, \beta$, and $p$, the random function $f$ is almost surely unbounded, and in particular discontinuous. Therefore showing that the appropriately constructed measure is ``concentrated'' in the set $Z(f) \cap \{f=y\}$ is a nontrivial task compared to \cite{Kahane}: to this end we exploit a technical device which could be useful in other contexts. Incidentally, this idea also streamlines Kahane's approach taking care of some annoying technicalities in his argument.

The next two sections will be dedicated to proving Theorem \ref{t:random-level-sets}.

\section{Measures concentrated on level sets and Berman's condition}

As already mentioned, since we have Theorem \ref{t:coarea-fractional} and Theorem \ref{t:random-Sobolev}, we focus on getting the inequality ${\rm dim}_H (\{f=y\}\cap Z(f)) \geq n-s$ in the two conclusions of Theorem \ref{t:random-level-sets}. We will use the following well-known criterion, which can be found in \cite[Chapter 8]{Mat95}.

\begin{lemma}\label{l:measure-criterion}
Assume $E\subset \mathbb T^n$ is a Borel set, $\sigma>0$, and $\mu$ a {\em nontrivial} measure concentrated on $E$ (i.e. $\mu (E^c)=0$) with the property that 
\[
I_\sigma [\mu] :=\int_{\T^n}\int_{\T^n}\frac{\,d\mu (x_1)\,d\mu (x_2)}{d(x_1,x_2)^{n-\sigma}} < \infty\, .
\]
Then the set $E$ has Hausdorff dimension at least $n-\sigma$. 
\end{lemma}

Therefore, given a level set $E_y := \{f=y\}\cap Z(f)$ for which we want to achieve the lower bound ${\rm dim}_H (\{f=y\}\cap Z(f))\geq n-s$, our goal is to show the existence of a suitable measure $\mu^y$ concentrated on $E_y$ which is nontrivial and for which the energy functional $I_\sigma [\mu^y]$ is finite for every $\sigma <s$. Kahane's idea is to produce $\mu$ as the limit of suitable ``regularizations''. More precisely, consider a function $\delta \in C^\infty_c (\mathbb R)$ with $\delta \geq 0$, $\int \delta =1$, and ${\rm spt}\, (\delta) \subset [-1,1]$. Let $\delta_\varepsilon (t) := \varepsilon^{-1} \delta (\varepsilon^{-1} t)$ and, given a measurable function $f$ on $\mathbb T^n$, define the measure
\begin{equation}\label{e:approx-mu}
\mu^y_\varepsilon := \delta_\varepsilon (f-y) \mathcal{L}^n\, ,
\end{equation}
where $\mathcal{L}^n$ denotes the Lebesgue measure. 
Observe that, since $\delta_\varepsilon (f-y)$ is a bounded function, the latter is {\em always} a finite measure. We next want to study the limit as $\varepsilon\downarrow 0$ of $\mu^y_\varepsilon$ for random functions $f$ and generic values $y\in \mathbb R$. Our analysis is split in two parts. The first is deterministic and is summarized in the following proposition.

\begin{proposition}\label{p:convergence-det}
Assume $f: \mathbb T^n \to \mathbb R$ is a measurable map such that $f_* \mathcal{L}^n$ is absolutely continuous. Then for a.e. $y\in \mathbb R$ there is a Radon measure $\mu^y$ such that
\begin{itemize}
\item[(i)] $\mu^y (\{f\neq y\})=0$; 
\item[(ii)] $\mu^y_\varepsilon \rightharpoonup^\star \mu^y$ as $\varepsilon\downarrow 0$.
\end{itemize}
Moreover the family of measures $\mu^y$ satisfies the following ``disintegration'' identity
\begin{equation}\label{e:disinteg}
\int_{\mathbb T^n} \varphi (x)\, dx = \int_{\mathbb R} \int \varphi (x) \, d\mu^y (x)\, dy
\qquad \forall \varphi \in C( \mathbb T^n)\, .
\end{equation}
\end{proposition}

Note that there will in general be many level sets for which the measure $\mu^y$ is in fact trivial. However \eqref{e:disinteg} guarantees that it will be nontrivial for some set $Y(f)$ of values with positive Lebesgue measure. 

The second tool is probabilistic and ensures that we can apply Proposition \ref{p:convergence-det} almost surely in the case of Theorem \ref{t:random-level-sets}.

\begin{lemma}\label{l:Berman-OK}
Consider the probability measure $\mathbb P$ of Theorem \ref{t:random-level-sets}. Then $f_* \mathcal{L}^n$ is absolutely continuous for $\mathbb P$-almost every $f$.
\end{lemma}

\subsection{Proof of Proposition \ref{p:convergence-det}}
We first recall the classical disintegration formula (cf. \cite[Theorem 5.3.1]{AGS}.  

\begin{theorem}[Disintegration theorem]\label{t:disintegration}
Let $X$ be a compact metric space with a Borel probability measure $\nu$. Let $Y$ be another measurable space and $f:X\to Y$ a Borel measurable map. Then there exists a family of Borel probability measures $\{\nu^y\}_{y\in Y}$ on $X$ with the following properties
\begin{itemize}
    \item[(i)] For $f_*\nu$-a.e. $y\in Y$, $\nu^y$ is concentrated on the fiber: $\nu^y(f^{-1}(y))=1$.
    \item[(ii)] For every bounded measurable function $\phi:X\to\R$, the function $y\mapsto\int_X\phi(x)\,d\nu^y(x)$ is a measurable function on $Y$, and we have the identity
    \[
    \int_X\phi(x)\,d\nu(x)=\int_Y \left(\int_X\phi(x)\,d\nu^y(x)\right)\,df_*\nu(y)\, .
    \]
\end{itemize}
Moreover, the family $\{\nu^y\}$ is unique in the following sense:
\begin{itemize}
    \item[(iii)] If $\{\tilde{\nu}^y\}$ is a another family satisfying (i) and (ii) above, then $\nu^y=\tilde{\nu}^y$ for $f_*\nu$-a.e. $y$.
\end{itemize}
\end{theorem}

We now apply Theorem \ref{t:disintegration} to the case of Proposition \ref{p:convergence-det}: $X= \mathbb T^n$, $Y= \mathbb R$, $\nu= \mathcal{L}^n$, $f: \mathbb T^n \to \mathbb R$, and $f_* \nu = \rho \mathcal{L}^1$ for some nonnegative integrable function $\rho$. Set therefore $\mu^y := \rho (y) \nu^y$. It is clear that $\mu^y$ satisfies both Proposition \ref{p:convergence-det}(i) and the identity \eqref{e:disinteg}. It suffices, therefore, to show Proposition \ref{p:convergence-det}(ii). To that end consider a countable dense set $\{\phi_j\}_{j=1}^\infty$ in $C^0(\mathbb T^n)$ and let also $\phi_0\equiv 1$. 

We plug the functions $\phi(x)=\phi_j(x)\delta_{\varepsilon}(f(x)-y_0)$ into the disintegration identity to obtain
\begin{equation}\label{e:disintegration-2}
\int_{\mathbb T^n} \phi_j(x)\,d\mu_{\varepsilon}^{y_0}=\int_\R (\int_{\mathbb T^n} \phi_j(x)\,d\mu^y(x))\delta_{\varepsilon}(y-y_0)\,dy\, .
\end{equation}
Let 
\[
g_j(y)=\int_{\mathbb T^n} \phi_j(x)\,d\mu^y(x)\, .
\]
Then $g_j\in L^1_y(\R)$ and the right hand side of \eqref{e:disintegration-2} is in fact the integral of the convolution $g_j*\delta_{\varepsilon}(\cdot - y_0)$. Now consider the set of points $y_0$ which are Lebesgue points for all $g_j$. These form a set of full (Lebesgue) measure and for any fixed one of them $y_0$ we have 
\[
\lim_{\varepsilon\to0}\int \phi_j(x)\,d\mu^{y_0}_{\varepsilon}=\int \phi_j(x)\,d\mu^{y_0}\, \qquad 
\mbox{for all $j$.}
\]
Finally, since $\phi_1=1$, we know $\mu_{\varepsilon}^{y_0}(\mathbb T^n)\to \mu^{y_0}(\mathbb T^n)$, so that the total mass of the measures is under control. In particular Proposition \ref{p:convergence-det}(ii) follows from the density of $\{\phi_j\}_{j\geq 1}$ and the weak$^\star$ compactness of bounded subsets in the space of Borel measures. 

\subsection{Proof of Lemma \ref{l:Berman-OK}} The argument is based on the following criterion introduced and proved by Berman, cf. \cite[Lemma 2.1]{Ber69}, in his study of local times of Gaussian processes.

\begin{theorem}[Berman's criterion]\label{t:Berman-crit}
Let $X$ be a metric space equipped with a finite (possibly signed) Borel measure $\mu$ and $f:X\to\R$ a Borel map satisfying the condition
\begin{equation}\label{e:Berman}
\int_\R\int_{X}\int_{X}e^{i\eta(f(x_1)-f(x_2))}\,d\mu(x_1)\,d\mu(x_2)\,d\eta<\infty\, .
\end{equation}
Then $f_*\mu\ll\mathcal{L}^1$ with a density $\rho\in L^1\cap L^2$. 
\end{theorem}

Since the proof is very concise and elegant, we include it for the reader's convenience. 

\begin{proof}
Fix any Schwarz function $\psi$ on $\R$, integrate it with respect to $f_*\mu$, and use the inversion formula on its Fourier transform $\hat\psi$ to write
\begin{equation}\label{e:we-will-Fubini-it}
\int_\R \psi(y)\,df_*\mu(y)=\int_X\psi(f(x))\,d\mu(x)=\frac{1}{\sqrt{2\pi}}\int_X \int_{\R}\hat{\psi}(\eta)e^{i\eta f(x)}\,d\eta\,d\mu(x)
\end{equation}
If we introduce the function 
\[
g (\eta) := \int_X e^{i\eta f(x)}\,d\mu(x)\, ,
\]
it is immediate to check that the left hand side of \eqref{e:Berman} is in fact $\|g\|_{L^2}^2$. Since this is finite, there is $\rho \in L^2(\R)$ such that $\hat{\rho}(\eta)= \frac{1}{\sqrt{2\pi}}g$. In particular, using Fubini and then Plancherel we can rewrite \eqref{e:we-will-Fubini-it} as 
\[
\int_\R \psi(y)\,df_*\mu(y)=\int_\R\psi(y)\rho(y)\,dy
\]
Since this holds for all Schwartz functions $\psi$, which are dense in $L^1 (\mathbb R)$, the conclusion of the lemma follows easily.
\end{proof}

\begin{proof}[Proof of Lemma \ref{l:Berman-OK}]
By Theorem \ref{t:Berman-crit}, it suffices to show
\begin{equation}\label{e:Berman-OK}
E:= \int \int_{\R}\int_{\T^n}\int_{\T^n} \exp \bigl(i\eta(f(x_1)-f(x_2))\bigr)\,dx_1\,dx_2\,d\eta\, d\mathbb P (f) <\infty
\end{equation}
First observe that 
\[
\int_{\T^n}\int_{\T^n} \exp \bigl(i\eta(f(x_1)-f(x_2))\bigr)\,dx_1\,dx_2=\left|\int_{\T^n} \exp\bigl(i\eta f(x)\bigr)\,dx\right|^2\ge0\, ,
\]
so we can exchange the integration in $f$ and $\eta$. Note moreover that the integrand is a complex function with modulus $1$ (given that $f$ is real-valued), and hence for every fixed $\eta$ it is absolutely integrable in the product measure $\mathcal{L}^n \times \mathcal{L}^n \times \mathbb P$. In particular we can again apply Fubini's theorem and compute first the expectation in $\mathbb P$, then integrate in $x_1$, $x_2$, and $\eta$.

Concerning the probability $\mathbb P$, given the independence of the random variables $B_{k, \lambda}$ and $W_{k,\lambda}$, we can represent it as a probability on the product space $\prod_{k,\lambda}\Omega^W_{k,\lambda}\times \Omega^B_{k,\lambda}$. Moreover, in order to simplify our notation, it is convenient to introduce the random variable $\tilde{B}_{k,\lambda} = 2^{-ks} (1- B_{k,\lambda}) + 2^{k\beta} B_{k, \lambda}$. Following this notation, and recalling that the $W_{k,\lambda}$ are normally distributed, we compute the expectation as
\begin{align*}
 &\E\bigr[\exp \bigl(i\eta(f(x_1)-f(x_2))\bigl)\bigr]\\
= &\prod_{k,\lambda} \int_{\Omega^B_{k,\lambda}} \int_{\Omega^W_{k,\lambda}}\exp \bigl(i\eta \tilde{B}_{k,\lambda}W_{k,\lambda}(\psi_{k,\lambda}(x_1)-\psi_{k,\lambda}(x_2))\bigr)\,dW_{k,\lambda}\,d\tilde{B}_{k, \lambda}\\
= &\prod_{k,\lambda} \int_{\Omega^B_{k,\lambda}} \exp \bigl(-\tfrac{1}{2}|\eta|^2 (\psi_{k,\lambda}(x_1)-\psi_{k,\lambda}(x_2))^2(\tilde{B}_{k,\lambda})^2\bigr)\,d\tilde{B}_{k,\lambda}\\
&\le  \exp \Big(-\tfrac{1}{2}|\eta|^2\sum_{k,\lambda}2^{-2sk}(\psi_{k,\lambda}(x_1)-\psi_{k,\lambda}(x_2))^2\Big)\, ,
\end{align*}
where we have used that $\tilde{B}_{k,\lambda} \geq 2^{-ks}$ almost surely.
Recall the separation property (S) in Theorem \ref{t:random-level-sets}: for every fixed $x_1\neq x_2$ there exists a pair $k,\lambda$ with
\begin{itemize}
\item $C^{-1} 2^{-k}\leq  d(x_1,x_2) \leq C 2^{-k}$,
\item $\psi_{k,\lambda}(x_2)=0$,
\item and $|\psi_{k,\lambda}(x_1)|\ge c$.
\end{itemize}
We see therefore that $\sum 2^{-2sk}(\psi_{k,\lambda}(x_1)-\psi_{k,\lambda}(x_2))^2\ge c\cdot d(x_1,x_2)^{2s}$ for all $x_1,x_2$, so we can obtain the bound
\[
E \leq \int \E\bigl[\exp \bigl(i\eta(f(x_1)-f(x_2))\bigr)\bigr]\, dx_1\, dx_2\, d\eta
\le \int \exp \bigl(-c|\eta|^2d(x_1,x_2)^{2s}\bigr)\, dx_1\, dx_2\, d\eta\, .
\]
Finally, we exchange the integrals once more and integrate in $\eta$ first to bound 
\[
E \leq C\int_{\T^n}\int_{\T^n}\frac{dx_1\,dx_2}{d(x_1,x_2)^s}<\infty
\]
This shows \eqref{e:Berman-OK} and thus completes the proof.
\end{proof}

\section{Kahane's estimate and proof of Theorem \ref{t:random-level-sets}}

Following Kahane's stragey, the last ingredient to prove Theorem \ref{t:random-level-sets} is given by the following estimate that we will prove later.

\begin{lemma}\label{l:Kahane}
Consider the probability measure $\mathbb P$ in Theorem \ref{t:random-level-sets} and define the measure $\mu^y_\varepsilon$ as in \eqref{e:approx-mu}. Then for every $\sigma >s$ there is a constant $C(\sigma)$, independent of $\varepsilon$, such that 
\begin{equation}\label{e:Kahane}
\int \mathbb E \left[I_\sigma (\mu^y_\varepsilon)\right]\, dy \leq C(\sigma)\, .
\end{equation}
Moreover, for every $y$ the limit
\begin{equation}\label{e:positive-limit}
\lim_{\varepsilon \downarrow 0} \mathbb E \left[\mu^y_\varepsilon (\mathbb T^n) \right]
\end{equation}
exists and it is a positive number $c(y)$. 
\end{lemma}

\subsection{Proof of Theorem \ref{t:random-level-sets}} By Theorem \ref{t:coarea-fractional} and Fubini, we already know that 
\[
{\rm dim}_H (\{f=y\}\cap Z(f)) \leq n-s \qquad \mbox{for $\mathcal{L} \times \mathbb P$-a.e. $(y,f)$.}
\]
We thus focus on the opposite inequality.

\medskip

Consider the first statement. Thanks to Lemma \ref{l:Berman-OK}, for $\mathbb P$-almost all $f$ we can apply Theorem \ref{t:Berman-crit} and Proposition \ref{p:convergence-det} conclude that there is a set $\bar{Y} (f)$ of full measure such that $\mu^y_\varepsilon \rightharpoonup^\star \mu^y$, where $\mu^y$ are the measures introduced in Proposition \ref{p:convergence-det}. Fix now any $\sigma>s$. Observe that, by lower semicontinuity of the energy $I [\mu]$ and Fatou we have 
\[
\int I_\sigma \left[\mu^y\right]\, dy
\leq \int \liminf_{\varepsilon\downarrow 0} I_\sigma \left[\mu^y_\varepsilon\right]\, dy
\leq \liminf_{\varepsilon\downarrow 0} \int I_\sigma \left[\mu^y_\varepsilon\right]\, dy\, ,
\]
while again by Fatou and in particular using Kahane's estimate \eqref{e:Kahane}, we conclude that 
\[
\int I_\sigma \left[\mu^y\right]\, dy < \infty
\]
for $\mathbb P$-a.e. $f$. Consider now the set $\tilde{Y} (f) \subset \bar{Y} (f)$ for which the measure $\mu^y$ is nontrivial and $I_\sigma [\mu^y]$. Since we can test the identity \eqref{e:disinteg} with $\varphi \equiv 1$, we conclude that $\tilde{Y} (f)$ must have positive measure. Observe also that, thanks again to \eqref{e:disinteg} $\mu^y ((Z(f))^c) =0$ for a.e. $y$, and thus we can also assume that $\mu^y$ is concentrated on $\{y=f\}\cap Z(f)$. Hence, applying Lemma \ref{l:measure-criterion} we conclude that the Hausdorff dimension of $\{y=f\}\cap Z(f)$ must be at least $n-\sigma$. We can now apply this argument for a sequence $\sigma_k\downarrow s$ to conclude that, for $\mathbb P$-almost every $f$ there is a set $Y(f)$ of positive measure such that ${\rm dinm}_H (Z(f)\cap \{y=f\})\geq n-s$.

\medskip

As for the second statement using Fubini and the argument above we know that, for a.e. $y$ there is a set $F_1(y)$ of probability $1$ with the properties that, for any $f\in F_1(y)$ the following three properties hold at at the same time:
\[
I_{s+n^{-1}} \bigl[\mu^y\bigr] < \infty \qquad \forall n \in \mathbb N\setminus \{0\}\, ,
\]
$\mu^y_\varepsilon \rightharpoonup^* \mu^y$, and $\mu^y$ is concentrated on $\{f=y\}\cap Z(f)$. Roughly speaking, we collect functions $f$ with a well-behaved measure $\mu^y$ that can be approximated by the measures $\mu^y_\epsilon$, and this happens almost surely.
Observe moreover that, since $(\mu^y_\varepsilon (\mathbb R^n))^2 \leq  I_{2s} \bigl[\mu^y_\varepsilon \bigr]$, we can also use Fatou and \eqref{e:Kahane} to conclude
\[
L:= \liminf_\varepsilon \mathbb E \left[ (\mu^y_\varepsilon (\mathbb R^n))^2\right]
< \infty\, .
\]
Consider now the random variable $\min\{\mu^y_{\epsilon}(\T^n),M\}$ for some large $M>0$ determined later. Notice that
\[
\E \left[\min\{\mu^y_\epsilon(\T^n),M\}\right]\ge \int_{\{\mu^y_\epsilon(\T^n)<M\}}\mu^y_\epsilon(\T^n)\,d\mathbb P\ge \E[\mu^y_\epsilon(\T^n)]-\frac{1}{M}\E\bigl[(\mu^y_\epsilon (\mathbb T^n))^2\bigr]\, .
\]
We now use again Fatou and the convergence of $\mu^y_\varepsilon (\mathbb T^n)$ to $\mu^y (\mathbb T^n)$ to get
\begin{align*}
\E \left[\min\{\mu^y(\T^n),M\}\right]
& =\int \limsup_{\epsilon\downarrow0}\min\{\mu^y_\epsilon(\T^n),M\}\, d\mathbb P (f)\\
&\geq \limsup_{\varepsilon \downarrow 0} \left(\E[\mu^y_\epsilon(\T^n)]-\frac{1}{M}\E\bigl[(\mu^y_\epsilon (\mathbb T^n))^2\bigr]\right)\\
& = \lim_{\varepsilon\downarrow 0} \E[\mu^y_\epsilon(\T^n)] -
\frac{1}{M}  \liminf_\varepsilon \mathbb E \left[ (\mu^y_\varepsilon (\mathbb R^n))^2\right]\\
&= c(y) - \frac{L}{M}\, ,
\end{align*}
where we used \eqref{e:positive-limit}. Since $c(y)$ is positive and we can choose $M$ large, we conclude that $\E \left[\min\{\mu^y(\T^n),M\}\right] >0$ for $M$ large enough. In other words, there is a set $F_2(y)$ of positive probability for which the measure $\mu^y$ is nontrivial. But then for every $f$ in the set $F(y) = F_1(y) \cap F_2(y)$ (which has the same probability as $F_2(y)$), we can apply Lemma \ref{l:measure-criterion} to conclude that the Hausdorff dimension of $\{f=y\}\cap Z(f)$ is at least $n-s$. 

\subsection{Proof of Lemma \ref{l:Kahane}} We use Fubini's Theorem and the inversion formula for the Fourier transform to write
\begin{align*}
E (y) := &\E \left [\int_{\T^n}\int_{\T^n}\frac{\,d\mu_\epsilon^y(x_1)\,d\mu_\epsilon^y(x_2)}{d(x_1,x_2)^{n-\sigma}}\right]\\
 = &C\int_{\T^n}\int_{\T^n}\Big(\int_\R\int_\R \E \bigl[e^{i\eta_1f(x_1)+i\eta_2f(x_2)}\bigr]e^{-i\eta_1y-i\eta_2y}\hat{\delta}(\epsilon\eta_1)\hat{\delta}(\epsilon\eta_2)\,d\eta_1\,d\eta_2\Big)\frac{\,dx_1\,dx_2}{d(x_1,x_2)^{n-\sigma}}.
\end{align*}
We now use the same setup as in the proof of Lemma \ref{l:Berman-OK} to compute further
\begin{align*}
& \E \bigl[ \exp \bigl(i\eta_1f(x_1)+i\eta_2f(x_2)\bigr)\bigr]\\
= & \prod_{k, \lambda} \int_{\Omega^B_{k,\lambda}} \exp \bigl(-\tfrac{1}{2} (\tilde{B}_{k,\lambda})^2(\eta_1\psi_{k,\lambda}(x_1)+\eta_2\psi_{k,\lambda}(x_2))^2\bigr) d\tilde{B}_{k,\lambda}\, .
\end{align*}
Thus, if we denote by $d \tilde{B}$ the probability in the product space $\Omega^B := \prod \Omega^B_{k, \lambda}$, 
the expectation $E$ can be written as
\[
E (y) =C\E_{\tilde{B}} \left[\int_{\T^n}\int_{\T^n}\left(\int_\R\int_\R e^{-\Phi }\hat{\delta}(\epsilon\eta_1)\hat{\delta}(\epsilon\eta_2)\,d\eta_1\,d\tilde{\eta}_2\right)\frac{\,dx_1\,dx_2}{d(x_1,x_2)^{n-\sigma}}\right]\]
where $\tilde{\eta}_2=\eta_1+\eta_2$, and
\begin{align*}
\Phi
= &\frac{1}{2}\left(\sum (\tilde{B}_{k,\lambda})^2(\psi_{k,\lambda}(x_1)-\psi_{k,\lambda}(x_2))^2\right)\eta_1^2\\
&+\left(\sum (\tilde{B}_{k,\lambda})^2\psi_{k,\lambda}(x_2)(\psi_{k,\lambda}(x_1)-\psi_{k,\lambda}(x_2))\right)\eta_1\tilde{\eta}_2\\
&+\frac{1}{2}\left(\sum (\tilde{B}_{k,\lambda})^2\psi_{k,\lambda}(x_2)^2\right)\tilde{\eta}_2^2\\
&+iy\tilde{\eta}_2\\
=: &\frac{1}{2}P\eta_1^2+Q\eta_1\tilde{\eta}_2+\frac{1}{2}R\tilde{\eta}_2^2+iy\tilde{\eta}_2.
\end{align*}
Recall the formula
\begin{equation}\label{e:domination}
\int_{\R}\int_{\R} e^{-\frac{1}{2}P\eta_1^2-Q\eta_1\tilde{\eta}_2-\frac{1}{2}R\tilde{\eta}_2^2-iy\tilde{\eta}_2}\,d\eta_1\,d\tilde{\eta}_2=\frac{2\pi}{\sqrt{PR-Q^2}}\exp \left(\frac{-Py^2}{2(PR-Q^2)}\right)\, .
\end{equation}
Then we integrate in $y$ and use Fubini and the bound $\|\hat{\delta}\|_\infty \leq C$ to estimate
\begin{align*}
\int \mathbb E \left[I_\sigma (\mu^y_\varepsilon)\right]\, dy &
\leq C \mathbb E_{\tilde B} \left[\int_{\mathbb T^n} \int_{\mathbb T^n} \int_{\mathbb R} \frac{2\pi}{\sqrt{PR-Q^2}}\exp \left(\frac{-Py^2}{2(PR-Q^2)}\right) \, dy\,  d(x_1, x_2)^{s-n}\, dx_1, dx_2
\right]\\
&= C \mathbb E_{\tilde B} \left[ \int_{\mathbb T^n} \int_{\mathbb T^n} \frac{d(x_1, x_2)^{\sigma-n}}{\sqrt{P}} \, dx_1\, dx_2 \right]\, .
\end{align*}
Recall now the separation property (S) of Theorem \ref{t:random-level-sets}: for every $x_1\neq x_2$ there exists a $k,\lambda$ with 
\begin{itemize}
    \item $C^{-1} d (x_1, x_2) \leq 2^{-k} \leq C d(x_1,x_2)$,
    \item $\psi_{k,\lambda}(x_2)=0$,
    and $|\psi_{k,\lambda}(x_1)|\ge c$. 
\end{itemize}
Using also $\tilde{B}_{k,\lambda}\ge 2^{-ks}$ a.s., we see that $P \ge c\cdot d(x_1,x_2)^{2s}$ for some positive constant $c>0$. Hence we can estimate
\[
\int \mathbb E \left[I_\sigma (\mu^y_\varepsilon)\right]\, dy
\leq  C \int_{\mathbb T^n} \int_{\mathbb T^n} d(x_1, x_2)^{\sigma-s-n}\, dx_1\, dx_2\, .
\]
Using that $\sigma>s$, we have therefore proved \eqref{e:Kahane}.

\medskip

For the second part of the lemma we use an entirely analogous strategy. Let us write 
\[
\mathbb E \left[\mu^y_\varepsilon (\mathbb T^n)\right]
= \int_{\T^n} \int_{\mathbb R} \mathbb E \left[ e^{i\eta f(x)}\right]\, e^{-i\eta y}\hat{\delta} (\varepsilon \eta) \, d\eta\, dx\, .
\]
We then can follow the same strategy as above splitting the expectation and computing the one with respect to the Gaussian random variables. We can thus introduce the function
$\Psi(x)=\sum (\tilde{B}_{k,\lambda})^2(\psi_{k,\lambda})^2(x)$, and use the separation property to show that $\Psi (x) \geq c$ almost surely. Using the dominated convergence theorem, $\hat{\delta} (0) =1$, and Fubini we arriving at computing  
\[\
\lim_{\varepsilon\downarrow 0}\E[\mu^y_\varepsilon(\T^n)] = C \int_{\T^n}\E_{\tilde{B}}\left[\frac{1}{\sqrt{\Psi(x)}}e^{-\frac{y^2}{2\Psi(x)}}\right]\,dx\, \, ,
\]
For the limit to be zero the function $\Psi (x)$ would have to be almost surely $\infty$, which is obviously not the case. 

\appendix

\section{Coarea inequality for H\"older functions}\label{a:coarea}

We present here an argument which uses \cite[Theorem 2.10.25]{Federer} as a black box and was pointed out to the first author by Bernd Kirchheim. \cite[Theorem 2.10.25]{Federer} implies the following (the statement is more general, but it is not needed for our purposes. If $(X,d)$ and $(Y, \delta)$ are two metric spaces, $f:X\to Y$ is a Lipschitz map and $Y$ is $\mathcal{H}^a$ $\sigma$-finite, then
\begin{equation}\label{e:general-metric}
\int \mathcal{H}^b (f^{-1} (\{y\}))\, d\mathcal{H}^a (y) \leq C (a,b) \bigl({\rm Lip\,} (f)\bigr)^a \mathcal{H}^{a+b} (X)\, .
\end{equation}
We now consider $Y= \mathbb R$ and we endow it with $\delta=e$ the standard Euclidean distance $e(t,\tau)=|t-\tau|$. Then we consider $X=Q_1\subset \mathbb R^n$ and endow it with the snowflake distance $d (p,q) := |p-q|^s$. On $\mathbb R^n$ as well we denote by $e$ the usual Euclidean distance $e(p,q)=|p-q|$. For subsets in $\mathbb R^n$ we let $\bar{\mathcal{H}}^k$ denote the Hausdorff measure relative to the metric $d$ and by $\mathcal{H}^k$ the Hausdorff measure with respect to the Euclidean distance. On $Y=\R$, the measure $\mathcal{H}^1$ coincides with the usual Lebesgue measure. 

From the definition we immediately see that 
\begin{equation}\label{e:metric-trick}
    \bar{\mathcal{H}}^k (A) = C (s, k) \mathcal{H}^{s k} (A) \qquad \forall A\subset \mathbb R^n\, .
\end{equation} 
In particular $\bar{\mathcal{H}}^{n/s} (Q_1)$ is finite. Moreover, if $f: Q_1 \to \mathbb R$ is $s$-H\"older continuous,
namely
\[
[f]_s := \sup_{x\neq y} \frac{|f(x)-f(y)|}{|x-y|^s}\, ,
\]
then $f: X \to Y$ is Lipschitz when we endow $X$ with the snowflake distance, and the Lipschitz constant $L$ of $f$ is exactly $[f]_s$. We can thus apply \eqref{e:general-metric} to get 
\[
\int \bar{\mathcal{H}}^{n/s -1} (f^{-1} (\{y\}))\, dy \leq C (n, s) [f]_s \bar{\mathcal{H}}^{n/s} (Y)\, .
\]
But then \eqref{e:metric-trick} implies \eqref{e:coarea-Hoelder}. 

\section{Proof of Theorem \texorpdfstring{\ref{t:random-Sobolev}}{random}}\label{a:random}

Throughout the proof the parameters $s,p$, and $\beta$ are fixed as in the statement. 
In order to ease our notation, we let $W$ be a random variable which is distributed according to the same law as all the $W_{k,\lambda}$. We use, moreover, the shorthand notation ${\rm Bin}\, (N, \gamma)$ for the binomial distribution, namely the law for the number of heads in $N$ independent coin tosses where probability of getting head in a single coin toss is $\gamma$. We recall moreover the Cherhoff bounds: if $X$ is the random variable counting the number of heads, then the probability $P$ of getting more than $(1+\sigma)\gamma N$ heads is estimated by
\begin{equation}\label{e:Chernoff}
P (\{X\geq (1+\sigma) \gamma N\}) \leq \left(\frac{e^\sigma}{(1+\sigma)^{1+\sigma}}\right)^{\gamma N}\, .
\end{equation}

\subsection{Proof of (i)} Fix $\alpha <s$. We will show
\begin{equation}\label{e:expectation-Wap}
\mathbb E \left[\|f\|_{W^{\alpha, p}}^p\right] < \infty\, ,
\end{equation}
from which (i) follows immediately.
Using Theorem \ref{t:PL-wavelets}, the linearity of the expectation, and the independence of the two random variables $W_{k,\lambda}$ and $B_{k, \lambda}$ for each pair $(k, \lambda)$, we can estimate 
\begin{align*}
\mathbb E \left[\|f\|_{W^{\alpha, p}}^p\right] &\leq C \mathbb E \left[ \|f\|^p_{\alpha, p} \right]
= C \sum_{k\in \mathbb N} \sum_{\lambda\in \Lambda_k} \mathbb E [2^{k (p \alpha -n)} |c_{k, \lambda}|^p]\\
&= C_p \sum_{k\in \mathbb N} 2^{k (p \alpha -n)} \sum_{\lambda \in \Lambda_k} [ 2^{-kp s} (1-2^{-\delta k}) + 2^{\beta k p} 2^{-\delta k}]\, ,
\end{align*}
where the constant $C_p$ depends on the expectation of $|W|^p$. 
Substituting the formula for $\delta$ (note in particular that $\delta$ is chosen so that $\beta p - \delta = -p s$), bounding $1-2^{-\delta k} \leq 1$, and considering that $|\Lambda_k|=2^{kn}$, we get 
\[
\mathbb E \left[\|f\|_{W^{\alpha, p}}^p\right] 
\leq C_p \sum_{k\in \mathbb N} 2^{kp (\alpha - s) + kn} |\Lambda_k|= C_p \sum_{k\in \mathbb N} 2^{-kp (s-\alpha)} < \infty\, .
\]

\subsection{Proof of (ii)} Fix $\alpha > s$. Consider a random function $f$, and for each $k$ let $S(k, f)$ be the set of $\lambda \in \Lambda_k$ such that $B_{k, \lambda} (f)=0$. Consider independent coin tosses where the probability of each toss being head is $2^{-k\delta}$. Then 
\[
p (k) := \mathbb P \left(\{f: |S(f,k)|< \tfrac{1}{2} |\Lambda_k|\}\right)
\]
is the probability that among $2^{kn}$ independent such coin tosses more than half of them gives head. The law of the cardinality of the complement set $L (f,k) = \Lambda_k \setminus S(f,k)$ follows then the binomial distribution ${\rm Bin} (2^{kn}, 2^{-\delta})$. In particular, using the Chernoff bound \eqref{e:Chernoff}, we infer that $p(k)$ decays at least exponentially in $k$. 
Therefore, by Borel-Cantelli, almost surely the set $S(k,f)$ has, for sufficiently large $k$, cardinality at least $2^{nk-1} = \frac{|\Lambda_k|}{2}$. Consider now 
\begin{equation}\label{e:lower-bound}
J (k,f) := \sum_{\lambda \in S(k,f)} 2^{k(p \alpha-n)} |c_{k,\lambda}|^p = \sum_{\lambda\in S(k,f)} 2^{kp (\alpha-s) -kn} |W_{k,\lambda}|^p\, .
\end{equation}
Recall that the $W_{k,\lambda}$ are independent identically distributed random variables with average $0$, variance $1$ and finite moments. Consider now the set 
\[
F (k_0) := \{f: |S(k,f)|\geq 2^{kn-1} \quad \forall k \geq k_0\}\, . 
\]
Using the independence of $B_{k,\lambda}$ and $W_{k, \lambda}$, Chebyshev's inequality gives  
\begin{align}
\mathbb P \left(  \left\{ f\in F(k_0) : \sum_{\lambda \in S (k,f)} |W_{k,\lambda}|^p < \frac{C_p}{2} |S (k,f)| \right\} \right) &\leq \frac{|S (k,f)| {\rm Var}\, (|W|^p)}{(C_p|S (k,f)|/2)^2} \label{e:Chebyshev}\\
&\leq \frac{\bar{C}}{|S (k,f)|} \leq \bar{C} 2^{-kn}\nonumber
\end{align}
(for an appropriate constant $\bar{C}$ depending only on $p$). 
In particular, again by the Borel-Cantelli Lemma, for $\mathbb P$-almost every $f$ the inequality
\[
 \sum_{\lambda \in S (k,f)} |W_{k,\lambda}|^p \geq C_p 2^{-(n-2)k}
\]
holds for all sufficiently large $k$. But, since $\bigcup_{k\geq k_0} F(k_0)$ is $\mathbb P$-almost all $L^1 (\mathbb T^n)$, the same holds for $\mathbb P$-a.e. $f\in L^1 (\mathbb T^n)$. Substituting in \eqref{e:lower-bound} we conclude that, almost surely, $J (k,f)$ is, for sufficiently large $k$'s, bounded from below by $c_p 2^{k (\alpha -s)}$ (where $c_p$ is a positive constant). Since by Theorem \ref{t:PL-wavelets} we have 
\[
\infty = \lim_{k\uparrow \infty} c_p 2^{k(\alpha-s)} \leq \lim_{k\uparrow \infty} J (k,f) \leq \|f\|_{\alpha, p} \leq C \|f\|_{W^{\alpha, p}}\, 
\]
for $\mathbb P$-almost every $f$.

\subsection{Proof of (iii)} Since $\mathbb T^n$ is bounded, $W^{\sigma, p} (\mathbb T^n) \subset W^{\sigma', p} (\mathbb T^n)$ for every $\sigma'< \sigma$. Therefore, it suffices to show that there is a function $\alpha\mapsto q(\alpha)$ defined on a suitable nonempty interval $(s_*,s)$ such that 
\begin{itemize}
\item[(a)] $\lim_{\alpha \uparrow s} q(s) = p$;
\item[(b)] for every $\alpha \in (s_*,s)$ $f\not\in W^{\alpha, q(\alpha)}$ almost surely.
\end{itemize}
We set $s_* = \max \{0, -\beta\}$ and 
\begin{equation}\label{e:choice-of-q}
q(\alpha) := p \, \frac{s + \beta}{\alpha + \beta} = \frac{\delta}{\alpha + \beta}\, .
\end{equation}
(a) is obvious, we therefore focus on $b$. Fix $\alpha\in (s_*,s)$ and $q=q(s)$. As in (ii), for every random $f$ we consider the set $L(k, f)= \Lambda_k\setminus S(k,f)$ of large outliers in $\lambda \in \Lambda_k$ to which the random variable $B_{k, \lambda}$ assigns the value $1$. Arguing as in (ii) and in particular using Chernoff bounds for the binomial distribution ${\rm Bin}\, (2^{kn}, 1- 2^{-\delta k})$ and Borel-Cantelli, we conclude that, almost surely, $|L (k, f)|\geq \frac{1}{2} 2^{(n-\delta) k}$ for all sufficiently large $k$. Consider now that 
\[
I (k, f) := \sum_{k\in L (k,f)} 2^{k(q\alpha-n)} |c_{k, \lambda}|^q 
\]
simplifies, thanks to our choice of $q$ and to the definition of $L(k,f)$, to
\[
I (k,f) = \sum_{k\in L (k,f)} 2^{k(\delta -n)} \sum_{k\in L (k,f)} |W_{k, \lambda}|^q\, .
\]
The same argument used for (ii), namely \eqref{e:Chebyshev} and Borel-Cantelli, implies that, if $C_q$ is the $q$-th moment of the random variable $W$, then, almost surely,  
\[
I (k,f) \geq 2^{k(\delta-n)} \frac{C_q}{2} |L (k,f)| \geq \frac{C_q}{4} 
\]
for all sufficiently large $k$. Again as for (ii), we invoke Theorem \ref{t:PL-wavelets} to conclude 
\[
\infty = \sum_k I (k,f) \leq C \|f\|_{W^{\alpha, q}}\, 
\] 
for $\mathbb P$-almost every $f$.

\section{Wavelets: multiresolution analysis and separation property}\label{a:multi-resolution}

In this section we will first give a digest of the construction of wavelets using multiresolution analysis (abbreviated as MRA). All the results are stated without proof and can be found in \cite{Dau92}. Then we will show, in Theorem \ref{t:separation}, that any basis of wavelets of mean-zero periodic $L^2$ functions, constructed as a tensorization of a one-dimensional MRA, satisfies the separation property (S) needed in Theorem \ref{t:random-level-sets}: for this claim we will provide a proof.

\subsection{MRA}\label{s:MRA}
We start by defining what an MRA is. To that respect, given any function $f$ over the real line, it is convenient to introduce the notation $f_{k,m}$ for the function
\[
f_{k,m} (x) = f (2^k x-m)\, .
\]
Even though the definition makes perfect sense for every pair of real numbers $k$ and $m$, in what follows both will always be taken among integers. 

\begin{definition}\label{d:MRA}
A MRA consists of a sequence of a countable family of closed subspaces $V_k\subset L^2(\R)$, indexed by $k\in \mathbb Z$, with the following properties:
\begin{itemize}
    \item[(i)] $V_{k-1}\subset V_k$ for every $k\in \mathbb Z$;
    \item[(ii)] $\overline{\cup_{k\in\Z}V_k}=L^2(\R)$;
    \item[(iii)] $\cap_{k\in\Z} V_k=\{0\}$;
    \item[(iv)] The spaces are related by dyadic scaling, namely $f\in V_k$ iff $f_{-k,0} \in V_0$;
    \item[(v)] $V_0$ is invariant under integer translations, namely if $f\in V_0$ then $f_{0,m}$ for all $m\in\Z$;
    \item[(vi)] There is a $\phi\in V_0$ such that $\{\phi_{0,m}\}_{m\in\Z}$ is an orthonormal basis for $V_0$. 
\end{itemize}
$\phi$ will be called the {\em scaling function}.
\end{definition}

With the following procedure it is possible to produce an ($L^\infty$ normalized) orthogonal wavelet basis using an MRA. Since $V_k\subset V_{k+1}$ is a closed subspace, there is an $L^2$ orthogonal complement $W_k$ of $V_k$ in $V_{k+1}$ and we can write $V_{k+1}=V_k\oplus W_k$.
The spaces $W_k$ are called the detailed spaces and it is clear from the definition that 
\[
L^2(\R)=\bigoplus_{k\in\Z}W_k
\]
and that 
\[
f\in W_k \quad \iff \quad f_{-k,0}\in W_0\, . 
\]
In fact, we can always find a $\chi$ such that $\{\chi_{0,m}\}_{m\in\Z}$ constitutes an orthonormal basis for $W_0$. This function $\chi$ will be called the {\em mother wavelet}. Thus, the set $\{\psi_{k,m};k\in\Z,m\in\Z\}$ forms an orthogonal basis for $L^2(\R)$.

Now, we use these functions to construct an orthogonal basis for $L^2(\mathbb T)$. Given a function $f$ on $\mathbb R$ we consider its periodization 
\begin{equation}\label{e:periodization}
f^{per}(x)=\sum_{\ell\in\Z}f(x-\ell)\, .
\end{equation}
If $f$ has sufficient decay at infinity, we can make sense of the right hand side \eqref{e:periodization} as an infinite series, but in fact we will periodize $\phi$ and $\chi$ and we will assume that their support is bounded, so that the corresponding sums in the right hand side of \eqref{e:periodization} will always have at most $N$ nonzero summands. An MRA with compactly supported functions $\phi$ and $\chi$ having $C^M$ regularity (for an arbitrary fixed $M\in \mathbb N$) is constructed in \cite{Dau88}.

It turns out that $\{\phi^{per}\}\cup\{\chi^{per}_{k,m}:k\ge 0, 0\le m\le 2^k-1\}$ forms an orthogonal basis for $L^2(\mathbb T)$. Moreover, one can show that $\phi^{per}\equiv 1$, so the family $\{\chi^{per}_{k,m}\}$ constitutes a basis for mean-zero functions on $\mathbb T$.

An higher dimension MRA can be obtained from a one-dimensional MRA using a suitable tensorization procedure. More precisely, for any $k$ consider the set $\lambda \in \Lambda_k$ which consists of any choice of a nonempty ordered $j$-tuple of integers $I :=\{1\leq i_1 < i_2 < \ldots < i_j \leq n\}$ and any choice of (not necessarily distinct) integers $M = \{m_1, \ldots , m_n\}$ with $m_\alpha \in \{0, \ldots, 2^k-1\}$. For any such $\lambda = (I, M)$ we define the function
\[
\psi_{k, \lambda} (x) = \prod_{i \in I} \chi^{per}_{k,m_i} (x_i) \prod_{\ell\not\in I} \phi^{per}_{k, m_\ell} (x_\ell)\, .
\]
Then the set $\{\psi_{k, \lambda}\}_{k\in \mathbb N, \lambda\in \Lambda_k}$ is an ($L^\infty$ normalized) orthogonal basis of wavelets for the subspace of mean zero functions of $L^2 (\mathbb T^n)$. 

\subsection{Separation property} We are now ready to show that wavelets constructed using tensorizations of MRAs satisfy the separation property (S) in Theorem \ref{t:random-level-sets}. 

\begin{theorem}\label{t:separation}
Consider an MRA for $L^2 (\mathbb R)$ with scaling function and mother wavelet which are both compactly supported and continuous. Let $\{\psi_{k, \lambda}\}$ be the orthogonal basis of wavelets of mean zero $L^2 (\mathbb T^n)$ obtained through the procedure described in Section \ref{s:MRA}. Then the basis satisfies the separation property (S) of Theorem \ref{t:random-level-sets}.
\end{theorem}

First of all we record the following lemma.

\begin{lemma}\label{l:separation}
Consider an MRA for $L^2 (\mathbb R)$ with compactly supported continuous mother wavelet $\chi$. Then there exists constants $c>0$ and $k_0\in \mathbb N$ with the following property. For every $x\in\R$, there exists a $0\le k\le k_0$ and an integer $j\in\Z$ such that $|\chi (2^kx-j)|\ge c$.
\end{lemma}
\begin{proof}
We know that the periodized functions $\psi_{k,m} := \chi^{per}_{k,m}$, $k\in \mathbb N$ and $0\le m\le 2^k-1$
form an orthogonal basis of the space of mean-zero 1-periodic $L^2$ functions. Choose the functions $f_1=\sin 2\pi x$ and $f_2=\cos 2\pi x$. Since $f_1^2+f_2^2=1$ we must have
$\max \{|f_1|, |f_2|\} \ge \frac{\sqrt{2}}{2}$. We next expand each in series of wavelets (which is possible because both functions have mean zero)
\[
f_i := \sum_{k=0}^\infty \sum_{0\leq m \leq 2^k-1} a_{k,m , i} \psi_{k, m}\, .
\]
Given the smoothness of $\sin$ and $\cos$, it is immediate to see that their wavelet expansions converge uniformly (in fact the uniform convergence holds for expansions of merely continuous functions, cf. \cite[Theorem 9.3.1]{Dau92}).
In particular we can find a $k_0\in \mathbb N$ with the property that, for every $x\in \mathbb T$, the inequality
\[
\sum_{k=0}^{k_0} \sum_{0\leq m \leq 2^k-1} |a_{k, m , i}| |\psi_{k, m} (x) |
\geq \left|\sum_{k=0}^{k_0} \sum_{0\leq m \leq 2^k-1} a_{k,m, i} \psi_{k, m} (x)\right| 
\geq \frac{\sqrt{2}}{4}\, 
\]
holds for at least one $i\in \{1,2\}$. Since this is the sum of $2^{k_0+1}-1$ numbers, if we denote by $M_0$ the maximum of the coefficients $|a_{k,m}|$ appearing in the finite sum and set $M= (2^{k_0+1} -1) M_0$, we conclude that, for every $x$, there is a $k\in \{0, \ldots , k_0\}$ and an $m\in \{0, \ldots, 2^{k-1}\}$ such that 
\[
|\psi_{k,m} (x)| \geq \frac{\sqrt{2}}{4M}\, .
\]
Recall that 
\[
\psi_{k,m} = \chi^{per}_{k,m} (x) = \sum_{\ell \in \mathbb Z} \chi (2^k (x-m) -\ell)\, ,
\]
and that for each $x$, $k$, and $m$ the number of $\ell$'s which are nonzero in the above summation is bounded by a constant $N$ independent of $x$, $k$, and $m$. This means that the inequality 
$|\chi (2^k x - (\ell + 2^k m))| \geq \frac{\sqrt{2}}{4MN}$ must hold for some choice of $\ell \in \mathbb Z$. Setting $c:= \frac{\sqrt{2}}{4MN}$ the claim of the lemma follows readily. 
\end{proof}

We are now ready to prove Theorem \ref{t:separation}. 

\begin{proof}[Proof of Theorem \ref{t:separation}]
Consider $x,y\in \mathbb T^d$ distinct. Identifying $\mathbb T^n$ with $\mathbb R^n/\mathbb Z^n$ we fix points $x_i, y_i\in [0,1)$ with the property that $x = (x_1, \ldots, x_n) + \mathbb Z^n$ and $y = (y_1, \ldots , y_n) + \mathbb Z^n$. Let $L$ be the maximum of the diameter of the support of $\chi$ and $\phi$ (the scaling and mother wavelet). Choose the integer $\bar{k}$ with the property that 
\begin{equation}\label{e:pick-max}
\frac{d(x,y)}{2\sqrt{n}} \leq 2^{-{\bar{k}}} L < \frac{d (x,y)}{\sqrt{n}}\, .
\end{equation}
We set $d(x_j, y_j) := \min \{ |x_j-y_j-m|: m \in \mathbb Z\}$ and 
pick a coordinate $i$ with the property that 
\[
d(x_i, y_i) = \max \{ d(x_j, y_j) : j\in \{1, \ldots , n\}\} \geq \frac{d(x,y)}{\sqrt{n}}
\]
We apply Lemma \ref{l:separation} to the point $\bar{x} = 2^{\bar k} x_i$ and we find, in particular, a $k\in \{0, k_0\}$ and an integer $\ell$ with the property that $|\chi (2^k \bar{x} - \ell)| \geq c$. Observe that, therefore
\begin{equation}\label{e:comparison}
\left|\chi \left(2^{\bar{k}+k} x_i - \ell\right)\right|\geq c\, .
\end{equation}
However \eqref{e:pick-max} implies also that $2^{-\bar{k}-k} L \leq 2^{-\bar{k}} L <1$. This means that, if we let $m_i \in \{0, \ldots, 2^{\bar k+k}-1\}$ be such that $m_i \equiv \ell \; mod \; (2^{\bar k+k})$, then  
\begin{equation}\label{e:first-bound}
|\chi^{per}_{\bar{k}+k, m_i} (x_i)| = |\chi (2^{\bar{k}+k} x_i - \ell)|\geq c\, .
\end{equation}
If we consider $\chi^{per}_{\bar{k}+k, m_i}$ as a function on the torus $\mathbb T$, this is $\psi_{\bar{k}+k, m_i}$ and the fact that $2^{-\bar{k}-k} L \leq 2^{-\bar{k}} L <1$ implies that the support of $\psi_{\bar{k}+k, m_i}$ (in $\mathbb T$) has the same diameter as the support of the function $\chi_{\bar{k}+k, m_i}$, which is at most $2^{-(k+\bar{k})} L$. Since the latter is strictly smaller than $d (x_i, y_i)$, we conclude then that 
\begin{equation}\label{e:vanishes}
|\chi^{per}_{\bar{k}+k, m_i} (y_i)|=0\, .
\end{equation}
We now claim, for each $j\neq i$, we can choose an apprioriate $m_j\in \{0, \ldots, 2^{\bar{k}+k}-1\}$ with the property that 
\begin{equation}\label{e:second-bound}
|\phi^{per}_{\bar{k}+k, m_j} (x_j)|\geq \bar{c}\, ,
\end{equation}
where $\bar{c}$ is a universal constant. Given that also the support of $\phi_{\bar{k}+k,0}$ has diameter smaller than $1$, this is equivalent to find an integer $\mu_j$ with the property that 
\[
|\phi (2^{\bar{k}+k} x_j + \mu_j)|\geq \bar{c}\, .
\]
But then the statement is equivalent to 
\begin{itemize}
\item[(P)] there is a constant $\bar{c}$ with the property that, for every $z\in \mathbb R$ there is an integer $\mu$ with $|\phi (z+\mu)|\geq \bar{c}$.  
\end{itemize}
Recall that 
\[
\sum_{\mu \in \mathbb Z} |\phi (z+\mu)| \geq \sum_{\mu \in \mathbb Z} \phi (z+\mu) = 1\, .
\]
On the other hand, given that the support of $\phi$ has diameter no larger than $L$, the number of summands in the left hand side which are positive is an integer smaller than $L+1$. In particular this shows that (P) holds if we choose $\bar{c} := \frac{1}{L+1}$. 

Consider now the wavelet $\psi_{\bar{k}+k, \lambda}$ given by 
\[
\psi_{\bar{k}+k, \lambda} (z) := \chi^{per}_{\bar{k}+k, m_i} (z_i) \prod_{j\neq i} \phi^{per}_{\bar{k}+k, m_j} (z_j)\, .
\]
Clearly \eqref{e:first-bound} and \eqref{e:second-bound} give $|\psi_{\bar{k}+k, m} (x)| \geq \bar{c} c^{n-1} >0$. Moreover \eqref{e:vanishes} gives $\psi_{\bar{k}+k, m} (y)= 0$. Finally, consider that $k\leq k_0$, where $k_0$ is the constant in Lemma \ref{l:separation}. The latter fact and \eqref{e:comparison} imply that 
\[
C^{-1} d(x,y) \leq 2^{-\bar{k}-k} \leq C d (x,y)\, ,
\]
for a constant $C$ which is independent of $x$ and $y$. 
\end{proof}

\providecommand{\bysame}{\leavevmode\hbox to3em{\hrulefill}\thinspace}
\providecommand{\MR}{\relax\ifhmode\unskip\space\fi MR }
\providecommand{\MRhref}[2]{%
  \href{http://www.ams.org/mathscinet-getitem?mr=#1}{#2}
}
\providecommand{\href}[2]{#2}


\end{document}